\documentclass[12pt, reqno]{amsart}

\usepackage{amsmath,amsthm,amsfonts,amscd,amssymb}
\usepackage{times,euler,eucal,eufrak,graphicx,color}
\usepackage[dvipdfmx]{hyperref}

\usepackage{caption} 
\usepackage{subcaption} 

\title[Weingarten calculus via orthogonality]
{Weingarten calculus via orthogonality relations: new applications}
\author {Beno\^\i{}t Collins}
\address{Department of Mathematics, Kyoto University} 
\email{collins@math.kyoto-u.ac.jp}
\author {Sho Matsumoto}
\address{Graduate School of Science and Engineering, Kagoshima University} 
\email{shom@sci.kagoshima-u.ac.jp}

\numberwithin{equation}{section}

\theoremstyle{plain}
\newtheorem{lemma}{Lemma}[section]
\newtheorem{theorem}[lemma]{Theorem}
\newtheorem{proposition}[lemma]{Proposition}

\theoremstyle{definition}
\newtheorem{definition}[lemma]{Definition}

\theoremstyle{remark}
\newtheorem{remark}[lemma]{Remark}
\newtheorem{example}[lemma]{Example}

\DeclareMathOperator{\Tr}{Tr}
\DeclareMathOperator{\Wg}{Wg}

\DeclareMathOperator{\U}{U}
\DeclareMathOperator{\Ort}{O}
\DeclareMathOperator{\Smp}{Sp} 
\DeclareMathOperator{\GL}{GL} 
\DeclareMathOperator{\COE}{COE}

\newcommand{\Sy}[1]{\mathcal{S}_{#1}}
\newcommand{\Pair}[1]{\mathcal{P}_2({#1})}

\newcommand{\C}{{\mathbb{C}}}

\newcommand{\mf}[1]{\mathfrak{#1}}

\newcommand{\mcal}[1]{\mathcal{#1}}

\begin{document}

\begin{abstract}
Weingarten calculus is a completely general and explicit method to compute the moments of the Haar measure on 
compact subgroups of matrix algebras. Particular cases of this calculus were initiated by theoretical physicists --
including Weingarten, after whom this calculus was coined by the first author, after investigating it
systematically. 
Substantial progress was achieved subsequently by the second author and coworkers, based on representation theoretic
and combinatorial techniques. All formulas of `Weingarten calculus' are in the spirit of Weingarten's seminal paper
\cite{weingarten}.
However, modern proofs are very different from Weingarten's initial ideas. 
In this paper, we revisit Weingarten's initial proof and we illustrate its power
by uncovering two new important applications: 
(i) a uniform bound on the Weingarten function, that subsumes existing uniform bounds, and is optimal up to a polynomial factor,
and (ii) an extension of Weingarten calculus to symmetric spaces and conceptual proofs of identities established by the second author.
\end{abstract}

\maketitle

\section{Introduction}

Let $G$ be a compact subgroup of $M_d(\C )$, $\mu$ its probability Haar measure,  and $u_{ij}: G\to \C$ the $ij$-coordinate function.
A \emph{Weingarten type} formula is a formula that computes 
$\int_G u_{i_1j_1}\ldots d\mu $. It is in general given as a sum over conditions on the indices $i, j$'s, of functions
called Weingarten functions. 
For example, in the case of unitary groups, the conditions are labeled by permutations, and in the orthogonal group, they 
are given by pair partitions. We refer to sections \ref{sec:Wg-unitary} and \ref{sec:Wg-orthogonal} for details. 

Weingarten's initial motivation (\cite{weingarten}) was to consider a sequence of subgroups $G_d$ of $M_d(\C )$, typically the unitary or 
orthogonal groups, and rather than describing precisely the integration formula, he was interested in the large $d$ asymptotics
of integrals. 

His basic observation was that the Weingarten functions satisfies a family of linear equations, and under an appropriate 
rescaling of the Weingarten functions by polynomials in $d$, they were satisfying a system that was upper triangular in the large
$d$ limit, and therefore, invertible for $d$ large enough. 
This approach was very slick, but one drawback was that it did not give precise information on the size of $d$ for which there 
was a unique solution. Until very recently, this idea has been dropped and replaced by equations arising from representation theory
\cite{collins-imrn,collins-sniady,matsumoto-novak} and others.
Another drawback is that the concrete methods to solve this linear system were not developed. 
The purpose of this paper is to revisit Weingarten's original approach, and address the aforementioned drawbacks. 

Specifically, the first author, together with Brannan, in \cite{brannan-collins}, realized recently that Weingarten's original 
approach was unavoidable when 
dealing with the asymptotics of Weingarten functions in the case of 
compact quantum groups, and could be improved into a very powerful and conceptual tool in this context. 

In  this paper, we revisit this tool in the context of the classical group. 
The formulas that one obtains are closely related to results obtained by Matsumoto and Novak in 
\cite{matsumoto-novak}, but they
are more elementary and more general, in the sense that no knowledge on 
Jucys-Murphy elements is needed, and that the technique can be adapted to more general cases.

Weingarten calculus has proven very useful in many situations, including free probability, random matrix theory, 
quantum information theory, representation theory, matrix integrals, and others (we refer to most of the
recent items of the bibliography for applications). 
One of the strength of this calculus is the very interesting properties of the Weingarten function in the large $d$ limit.
For example, in the unitary case,
$$\Wg (\sigma , d) = d^{-k-|\sigma |} \mathrm{Moeb} (\sigma )$$
We refer for section \ref{section:uniform_bounds} for notation.  

In some cases, it is desirable to obtain uniform estimates on Wg. Such bounds have been obtained by the first author and other 
coauthors in \cite[Theorem 4.1]{CGGPG}, and also by \cite[Lemma 16]{montanaro}.

In a related domain, in random matrix theory, moment methods are a widely used and powerful tool. 
On the one hand, the proofs of asymptotic freeness of unitarily invariant random matrix under the weakest possible assumptions are
obtained with Weingarten techniques (\cite{collins-imrn}). 
On the other hand, uniform moment estimates of a power or a random matrix degree depending on the dimension, and high enough (typically 
much higher than the logarithm of the dimension)
give norm convergence estimates. 
We refer to \cite{soshnikov} 
for one of the first seminal applications of this method to random matrix theory.

Putting these two observations together, it is very natural to try to achieve a fine and uniform convergence of 
the behaviour of the Weingarten function as the dimension goes to infinity, and the size of the permutation group too. 
This problem of finding a uniform estimate has also applications in more unexpected fields, such as
Quantum Information Theory, cf e.g. \cite{montanaro, CGGPG} for weaker uniform bounds with specific applications. 
See also  \cite{berkolaiko-kuipers}.

In this respect, our main result is a uniform bound within a polynomial factor, which is obtained in 
Theorems \ref{wg-estimate}, \ref{wg-symplectic-bound}, and \ref{wg-orthogonal-bound}
and which we record in the theorem below just in the case of the unitary group -- the other cases covered in this paper showcase results
of similar flavour (albeit with different proofs)
\begin{theorem}
For any $\sigma\in \Sy{k}$ and  
$d > \sqrt{6} k^{7/4}$, 
$$\frac{1}{1-\frac{k-1}{d^2}} \le 
\frac{d^{k+|\sigma|} \Wg^{\U} (\sigma,d)}{\mathrm{Moeb}(\sigma)}  
\le \frac{1}{1- \frac{6 k^{7/2}}{d^2}}.
$$
In addition, the l.h.s inequality is valid for any $d \ge k$.
\end{theorem}

Let us note that this revisited approach to Weingarten calculus 
is related to, and implies results of
\cite{matsumoto-novak} in the unitary case
and from \cite{zinn-justin, matsumoto-ramanujan} in the orthogonal case.

Finally, Weingarten calculus extends beyond groups, to the context of symmetric spaces \cite{cartan}. 
Although push forward allow in principle to 
compute any Haar measure on a symmetric space \cite{collins-stolz}, 
the second author observed some phenomena intrinsic to some classes \cite{matsumoto-rm1,matsumoto-rm2}.
These phenomena were obtained by computation  without conceptual explanation. 
It turns out that in some cases, Weingarten's original approach supplies this conceptual explanation. 
This is the content of theorems \ref{thm:COE} and \ref{thm:AIII}.

This paper is organized as follows: 
After this introduction, section 2 revisits and conceptualizes Weingarten's original integration technique in the unitary context.
Section 3 uses section 2 to provide the best uniform bounds known so far. 
Section 4 handles sections 2 and 3 in the context of orthogonal and symplectic cases. 
Section 5 develops the Weingarten calculus on symmetric spaces.

\subsection*{Acknowledgements}
Both authors were supported by 
JSPS KAKENHI Grant Numbers 26800048, 25800062.   
They would like to thank an anonymous referee for very constructive comments
on the first version of the manuscript.  
BC acknowledges useful discussions with Mike Brannan.

\section{Unitary groups}

\subsection{Weingarten calculus}\label{sec:Wg-unitary}

Throughout this section, we suppose $d,k$ are positive integers with $d \ge k$.
For each permutation $\sigma \in \Sy{k}$, the {\em unitary Weingarten function}
$\Wg^{\U}(\sigma,d)$ is, by definition,
\begin{equation} \label{eq:definition-unitary-wg}
\Wg^{\U} (\sigma, d) = 
\int_{\U(d)} u_{11} u_{22} \cdots u_{kk}
\overline{u_{\sigma(1)1} u_{\sigma(2)2} \cdots u_{\sigma(k)k}} \, d\mu,
\end{equation}
where $d\mu = d\mu^{\U (d)}$ denotes the normalized Haar measure on $\U(d)$.
It is easy to see that
the function $\sigma \mapsto \Wg^{\U} (\sigma,d)$ is conjugacy-invariant, i.e.,
$$
\Wg^{\U} (\tau^{-1} \sigma \tau,d) =\Wg^{\U} ( \sigma,d) 
\qquad \text{for any $\sigma,\tau \in \Sy{k}$}.
$$
The Weingarten calculus for $\U(d)$ is described as follows.

\begin{lemma}[\cite{collins-imrn}]
\label{lemma:unitaryWC}
For four sequences 
$$
\mathbf{i}=(i_1,\dots,i_k), \quad
\mathbf{i}'=(i_1',\dots,i_k'), \quad 
\mathbf{j}=(j_1,\dots,j_k), \quad
\mathbf{j}'=(j_1',\dots,j_k')
$$
of positive integers in $\{1,2,\dots,d\}$, we have
\begin{align*}
& \int_{\U(d)} u_{i_1, j_1} u_{i_2, j_2} \cdots u_{i_k, j_k}
\overline{u_{i_1', j_1'} u_{i_2', j_2'} \cdots u_{i_k', j_k'}} \, d\mu \\
=& \sum_{\sigma \in \Sy{k}}
\sum_{\tau \in \Sy{k}} 
\delta_{\sigma} (\mathbf{i}, \mathbf{i}')
\delta_{\tau} (\mathbf{j}, \mathbf{j}')
\Wg^{\U}(\sigma\tau^{-1} ,d).
\end{align*}
Here $\delta_{\sigma} (\mathbf{i}, \mathbf{i}')$ is given by
$$
\delta_{\sigma} (\mathbf{i}, \mathbf{i}') =
\begin{cases}
1 & \text{if $i_{\sigma(r)} = i'_r$ for all $r$}, \\
0 & \text{otherwise}.
\end{cases}
$$
\end{lemma}

\subsection{Orthogonality relations}

We give orthogonality relations for  Weingarten functions $\Wg^{\U}(\cdot,d)$,
which comes from the orthogonal (or unitary) property of the random matrix $U$ itself.
This is first found in \cite{samuel}. See also \cite{GGPN} and its references.

\begin{proposition} 
For any $\sigma \in \Sy{k}$, we have
\begin{equation} \label{eq:recurrence-unitaryWg}
d \Wg^{\U}(\sigma,d)= - \sum_{i=1}^{k-1} \Wg^{\U} ((i,k) \sigma,d)+ 
\delta_{\sigma(k)=k} \Wg^{\U}(\sigma^{\downarrow},d). 
\end{equation}
Here $\sigma^{\downarrow} \in \Sy{k-1}$ is the restriction of $\sigma$ 
to the permutation on the set $\{1,2,\dots,k-1\}$
(if $\sigma(k)=k$)
 and $(i,k)$ is the transposition 
between $i$  and $k$.
Moreover, $\delta_{\sigma(k)=k}$ equals to $1$ if $\sigma$ fixes $k$, and 
to $0$ otherwise.
\end{proposition}

\begin{proof}
Consider the sum of integrals
\begin{equation} \label{eq:sum_unitary}
\sum_{i=1}^d \int_{\U(d)} u_{11} \cdots u_{k-1,k-1} u_{k,i}
\overline{ u_{\sigma(1) 1} \cdots u_{\sigma(k-1),k-1} u_{\sigma(k),i}} d\mu.
\end{equation}
Since a matrix $U=(u_{ij})$ is unitary, we have $\sum_{i=1}^d u_{k,i} \overline{u_{\sigma(k),i}}
= \delta_{\sigma(k)=k}$
and therefore it equals
\begin{align}
& \delta_{\sigma(k)=k}\int_{\U(d)} u_{11} \cdots u_{k-1,k-1}
\overline{ u_{\sigma(1) 1} \cdots u_{\sigma(k-1),k-1} } d\mu \notag \\
=&
\delta_{\sigma(k)=k} \Wg^{\U} (\sigma^{\downarrow},d). 
\label{eq:sum_unitary2}
\end{align}

On the other hand, using Lemma \ref{lemma:unitaryWC} we have
\begin{align*}
& \int_{\U(d)} u_{11} \cdots u_{k-1,k-1} u_{k,i}
\overline{ u_{\sigma(1) 1} \cdots u_{\sigma(k-1),k-1} u_{\sigma(k),i}} d\mu \\
=& 
\begin{cases} 
\Wg^{\U} (\sigma,d) & \text{if $i \ge k$}, \\
\Wg^{\U} (\sigma,d)+\Wg^{\U} ((i,k)\sigma,d)  & \text{if $i < k$.}
\end{cases}
\end{align*}
In fact, in the notation of Lemma \ref{lemma:unitaryWC}, 
the delta symbol $\delta_{\tau}(\mathbf{j}, \mathbf{j}')$
with $\mathbf{j}=\mathbf{j}'=(1,\dots,k-1,i)$ survives only if
$\tau$ is the identity permutation or
the transposition $(i,k)$ with $i<k$.
Summing up them over $i$, we obtain 
$$
\eqref{eq:sum_unitary} = d \Wg^{\U}(\sigma,d)+ \sum_{i=1}^{k-1} \Wg^{\U}(
(i,k) \sigma,d)
$$ 
Combining this with \eqref{eq:sum_unitary2}, we obtain the proposition.
\end{proof}

\begin{example}
We use the one-row notation $[\sigma(1),\sigma(2), \dots, \sigma(k)]$
for $\sigma \in \Sy{k}$.
The relation \eqref{eq:recurrence-unitaryWg} with $k=1$ and $\sigma= [1] \in \Sy{1}$
gives the relation $d \Wg^{\U}([1],d)= \Wg^{\U} (\emptyset,d)=1$, and hence
$\Wg^{\U}([1],d)=\frac{1}{d}$.
Furthermore, for $k=2$ we find 
\begin{align*}
d \Wg^{\U}([1,2],d)=& - \Wg^{\U}([2,1],d)+ \Wg^{\U} ([1],d) \\
d \Wg^{\U}([2,1],d)=& - \Wg^{\U}([1,2],d).
\end{align*} 
Solving this linear system of equations, we obtain
$$
\Wg^{\U}([1,2],d)= \frac{1}{d^2-1}, \qquad 
\Wg^{\U}([2,1],d)= \frac{-1}{d(d^2-1)}.
$$ 
\end{example}

\subsection{Weingarten graphs} \label{subsection:Weingarten-graphs}

\begin{definition}
We define an infinite directed graph $\mathcal{G}^{\U}=(V,E)$ as follows. 
\begin{itemize}
\item The vertex set $V$ is $\bigsqcup_{k=0}^\infty \Sy{k}$.
Each vertex $v$ in $\Sy{k}$ is said to be of {\em level} $k$.
The vertex $\emptyset \in \Sy{0}$ is the unique element of level $0$.
\item For each $k \ge 2$, 
two vertices $\sigma, \tau$ of level $k$ (i.e. in $\Sy{k}$) 
are connected by a {\em solid arrow} as $\sigma \rightarrow \tau $ if and only if 
$$
\tau = (i,k) \sigma \quad \text{with some $i$ smaller than $k$}.
$$
\item For each $k \ge 1$, 
a vertex $\sigma$ of level $k$
and a vertex $\sigma'$ of level $k-1$ are connected by a {\em dashed arrow}
as $\sigma \dashrightarrow \sigma'$ 
if and only if $\sigma(k)=k$ and 
$\sigma'=\sigma^{\downarrow}$.  
\end{itemize}
We call $\mathcal{G}^{\U}$ the {\em Weingarten graph} for the unitary group.
\end{definition}

Each vertex $\sigma$ of level $k$ is connected by exactly $k-1$ solid arrows and
radiates at most $1$ dashed arrow 
to $\sigma^{\downarrow}$ if it exists.

Let $\sigma$ be a vertex of level $k$.
A sequence 
$p= (\sigma_0, \sigma_1, \dots, \sigma_{l+k})$ of
vertices is called 
a {\em path} from $\sigma$ to $\emptyset$ 
of length $l+k$ if
$\sigma_0=\sigma$, $\sigma_{l+k}=\emptyset$, and, 
for each $i=1,2,\dots,l+k$, vertices $\sigma_{i-1}$ and $\sigma_{i}$
are connected by an edge.
Since only dashed arrows lower levels of vertices,
a path $p$ always goes through exactly $k$ dashed arrows, 
and hence the length of $p$ is at least $k$. 
Denote by $P(\sigma,l)$ the collection of such paths. Especially, 
every path $p \in P(\sigma,l)$ goes through exactly $l$ solid edges.

\begin{lemma} \label{lemma:unitary-Wg-expansion1}
For $\sigma \in \Sy{k}$, we have the  
expansion
\begin{equation} \label{eq:unitary-Wg-expansion1}
\Wg^{\U}(\sigma,d)= d^{-k}  \sum_{l \ge 0} \# P(\sigma, l) (-d^{-1})^{l}. 
\end{equation}
\end{lemma}

\begin{proof}
The relation \eqref{eq:recurrence-unitaryWg} is 
expressed as 
$$
\Wg^{\U}(\sigma,d)= \sum_{\tau: \sigma \rightarrow 
\tau} \Wg^{\U}(\tau,d) (-d^{-1}) + 
\delta_{\sigma(k)=k} \Wg^{\U}(\sigma^{\downarrow},d) d^{-1},
$$
where the sum of the right hand side runs over
$\tau \in \Sy{k}$ connected with $\sigma$ by 
a solid arrow.

We consider the infinite dimensional vector space spanned by the basis $V$ and we denote it $\C^V$.
We call $\delta_\sigma, \sigma\in  \Sy{k}$. 
On $\C^V$ we introduce the linear map that is the linear extension of 
$$
Q(\delta_\sigma)= 
\begin{cases}
 \sum_{\tau: \sigma \rightarrow 
\tau} (-d^{-1})\delta_\tau + 
\delta_{\sigma(k)=k} d^{-1} \delta_{\sigma^{\downarrow}} 
& \text{if $\sigma \not=\emptyset$} \\  
\delta_{\emptyset} & \text{if $\sigma=\emptyset$.}
\end{cases}
$$

In addition, let us introduce the linear form $\Wg: \C^V\to \C$ given by 
$\Wg (\delta_\sigma ) = \Wg^{\U}(\sigma,d)$.
It follows directly from equation  \eqref{eq:recurrence-unitaryWg}
that 
$$\Wg\circ Q=\Wg .$$ 

Note that $Q (\delta_\emptyset ) =\delta_\emptyset$ but for any other basis element, $Q$ has the effect of multiplying by $d^{-1}$
times a vector whose coordinates do not depend on $d$.
More precisely, if we view $Q$ formally as an endomorphism of $\C [[d^{-1}]]^V$, we can write it as
$$Q=P+d^{-1} T$$ where $P,T$ are endomorphisms of $\C^V$ (that act naturally on $\C [[d^{-1}]]^V$), $P$ is the rank one 
projection whose range is $\C \delta_\emptyset$ and 
whose kernel is the span of all remaining canonical basis elements. 

Therefore, it follows directly that $Q^{\circ l} (\delta_\sigma )$ converges formally as $l\to\infty$ (in the sense that
each coefficient $Q^{\circ l}(\delta_\sigma)_\tau$
(of $Q^{\circ l}(\delta_\sigma)$) viewed
as a rational fraction in $d^{-1}$ converges pointwise -- specifically, the term
of degree $p$ becomes steady as soon as $l>p$ because of the structure of $Q$.
In this sense, we can define the limit of $Q^{\circ l}$ as $l\to\infty$ as an element of 
$End (\C [[d^{-1}]]^V)$, that we will call $Q^{\circ\infty }$.

If one knows beforehand that $\Wg^{\U}(\sigma,d)$ can be seen as a power series in $d^{-1}$ (this is the case because it is 
rational fraction) then the proof is complete by considering the limit $Q^{\circ\infty}$ of $Q^{\circ l}$ as $l\to \infty$,
and the equation
$\Wg\circ Q^{\circ\infty } = \Wg$ at $\delta_\sigma$. 

However, for the sake of completeness and of obtaining more information on $\Wg$, we add one analytic proof that does
not require the knowledge that $\Wg^{\U}(\sigma,d)$ is a power series in $d^{-1}$.

For this, we introduce the subset $V_n$ of $V$ as the finite disjoint union of the $n+1$ first symmetric groups
 $V_n = \bigsqcup_{k=0}^n \Sy{k}$.
 It is clear that $\Wg$ can be defined on $\C^{V_n}\subset \C^V$ and that $Q$ leaves $\C^{V_n}$ invariant, and 
 that on $\C^{V_n}$, $\Wg\circ Q^{\circ\infty } = \Wg$ remains true 
 on $\C^{V_n}$, we introduce the $l_1$-type  norm $\| \sum \alpha_\tau \delta_\tau \|= \sum |\alpha_\tau|$.
 
 We use the notation $P$ for the projection introduced earlier in the first part of this proof, and we note that
 $Q\circ P = P$. Next, we introduce the notation $R= Q\circ (1-P)$. Note that $Q=P+R$.
 By inspecting equation
    \eqref{eq:recurrence-unitaryWg} one sees that
    $$\|R (x)\| \le nd^{-1}\|x\|.$$
 Iterating, for any integer $l\ge 1$, 
 $\|R^{\circ l} (x)\| \le n^ld^{-l}\|x\|.$
 One checks by induction that for any integer $l>1$, 
 $$Q^{\circ l}\circ (1-P)= P\circ (R+R^{\circ 2}+\ldots +R^{\circ l-1})+ R^{\circ k}.$$
Therefore,
$$Q^{\circ l}= P\circ (1+ R+R^{\circ 2}+\ldots +R^{\circ l-1})+ R^{\circ k}.$$
The inequality $\|R^{\circ l} (x)\| \le n^ld^{-l}\|x\|$ implies
$\|P\circ R^{\circ l} (x)\| \le n^ld^{-l}\|x\|,$
therefore $Q^{\circ l}$ converges with respect to any norm topology (as they are all equivalent in finite dimension).

Calling again its limit $Q^{\circ\infty }$, we conclude that 
$\Wg\circ Q^{\circ\infty } = \Wg$, apply this equality at  $\delta_\sigma$, and conclude as in the formal case.
\end{proof}

\begin{remark}
As a byproduct of the analytic proof presented above, we obtain a bound on the Weingarten function
for any $d>n$. This bound is refined and uniformized subsequently in this paper.
\end{remark}

For each permutation $\sigma \in \Sy{k}$, 
we associate with the cycle-type $\mu$, 
which is an integer partition of $k$.
Put $|\sigma|=k-\ell(\mu) \in \{0,1,\dots,k-1\}$,
where $\ell(\mu)$ is the length of $\mu$.
For example, 
$$
|\sigma| =  \begin{cases}
0 & \text{if $\sigma$ is the identity permutation}, \\
1 & \text{if $\sigma$ is a transposition},
\end{cases}
$$
and $|\sigma| \ge 2$ otherwise.

For any transposition $(i,j)$, we find that
$|(i,j) \sigma|$ is equal to $|\sigma|\pm 1$.  
Moreover, if $\sigma(k)=k$, then 
$|\sigma^{\downarrow}|=|\sigma|$.
In other words, 
for a path $p=(\sigma_0, \sigma_1, \dots, \sigma_{l+k})$
in $P(\sigma,l)$, we see that
$$
|\sigma_i|- |\sigma_{i-1}| = 
\begin{cases} 
+1 \ \text{or} \ -1  & \text{if $\sigma_{i-1} \longrightarrow \sigma_i$},\\
0 & \text{if $\sigma_{i-1} \dashrightarrow \sigma_i$}.
\end{cases}
$$
Since $|\sigma_{l+k}|=|\emptyset|=0$, we 
find  
$\# P(\sigma,l) =0$ unless
$l=|\sigma|+2g$ with some integer $g \ge 0$.
We call this property a {\em parity condition} for path $p$
(or for $\sigma$).
The expansion \eqref{eq:unitary-Wg-expansion1} can be now
reformulated as follows.

\begin{theorem} \label{thm:unitary-Wg-expansion2}
For each $\sigma \in \Sy{k}$, we have the formal expansion
$$
(-1)^{|\sigma|} d^{|\sigma|+k}\Wg^{\U}(\sigma,d)=
\sum_{g \ge 0} \# P(\sigma, |\sigma|+2g) d^{-2g}.
$$
\end{theorem}

\subsection{Monotone factorizations}

Consider a permutation $\sigma \in \Sy{k}$ and 
a sequence
$f= (\tau_1,\dots, \tau_l)$
of $l$ transpositions satisfying:
\begin{itemize}
\item $\tau_i =(s_i,t_i)$ with $1 \le s_i <t_i \le k$;
\item $k \ge t_1  \ge \cdots \ge t_l \ge 1$;
\item $\sigma=\tau_1 \cdots \tau_l$.
\end{itemize}
Such a sequence $f$ is called a {\em monotone factorization} of length $l$ 
for $\sigma$.
We denote by $\mathcal{F}(\sigma,l)$ the collection of 
these $f$.
Paths in $P(\sigma,l)$  are 
naturally identified with monotone factorizations.

\begin{lemma} \label{lemma:PvsF}
Let $\sigma \in \Sy{k}$. For any nonnegative integer $l$, 
there exists a $1$-to-$1$ correspondence between 
$P(\sigma,l)$ and $\mathcal{F}(\sigma,l)$.
\end{lemma}

\begin{proof}
First, we construct a correspondence $P(\sigma,l) \to \mathcal{F}(\sigma,l)$.
Pick up a path $p=(\sigma_0,\sigma_1,\dots, \sigma_{k+l})$
in $P(\sigma,l)$.
There exist $l$ solid arrows in $p$:
$$
\sigma_{i_j-1} \rightarrow \sigma_{i_j},
\quad  \text{where $1 \le i_1 < i_2 <\cdots <
i_{l} \le k+l$}.
$$
Each solid arrow $\sigma_{i_j-1} \rightarrow \sigma_{i_j}$ associates with a transposition $\tau_j=(s_j,t_j)$ satisfying
$\sigma_{i_j}= (s_j,t_j) \sigma_{i_j-1}$, 
where $t_j$ is the level of $\sigma_{i_j-1}$.
Then $p$ gives the relation
$e_k= \tau_l \cdots \tau_1 \sigma$, or equivalently
$\sigma=\tau_1 \cdots \tau_l$.
Since levels of $\sigma_i$ are weakly decreasing, $t_j$ are as well.
Thus we have obtained a monotone factorization 
$f=(\tau_1,\dots, \tau_l)$ in $\mathcal{F}(\sigma,l)$.

Next we construct the inverse map $\mathcal{F}(\sigma,l) \to P(\sigma,l)$.
Let $f=(\tau_1,\dots,\tau_l)$ be a monotone factorization for $\sigma$.
Set $\sigma_0:=\sigma$ and consider $\tau_1=(s_1,t_1)$.
\begin{itemize}
\item If $t_1=k$, then 
we put $\sigma_1:= \tau_1 \sigma_0$ and 
deal with a solid arrow $\sigma_0 \rightarrow \sigma_1$.
\item Assume $t_1=k-r$ with $r \ge 1$. Then
the monotonicity for $f$ forces 
$\sigma_0(s)=s$ for $k-r+1 \le s \le k$, and 
we can put $\sigma_{i}=(\sigma_{i-1})^{\downarrow}
\in \Sy{k-i}$
 ($i=1,2,\dots,r$) inductively,
and finally $\sigma_{r+1}= \tau_1 \sigma_r \in \Sy{k-r}$.
We thus have obtained
a ``partial path''
$$
\sigma_0 \dashrightarrow \sigma_1 \dashrightarrow
\cdots \dashrightarrow \sigma_{r} \rightarrow \sigma_{r+1}.
$$
\end{itemize}
If we repeat this operation for the end point $\sigma_{r+1}$ $(r \ge 0)$ and for 
the next
$\tau_j$ $(j=2,\dots,l)$
until it arrives at $\emptyset$, 
we can construct a path $p=(\sigma_0, \dots, \sigma_{k+l})
\in P(\sigma,l)$.

It is clear that the above two correspondences $p \mapsto f$ and $f \mapsto p$
are inverse each other.
\end{proof}

\begin{example}
The following objects are identified by the previous lemma.
\begin{itemize}
\item A path in $P([4,1,5,3,2],4)$:
\begin{align*}
&[4,1,5,3,2] \xrightarrow{(3,5)} [4,1,3,5,2] \xrightarrow{(2,5)}
[4,1,3,2,5] \dasharrow [4,1,3,2] \\
& \xrightarrow{(2,4)} 
[2,1,3,4] \dasharrow [2,1,3] \dasharrow [2,1]  
\xrightarrow{(1,2)} [1,2] \dasharrow [1] \dasharrow \emptyset.
\end{align*}
\item a monotone factorization in $\mathcal{F}([4,1,5,3,2],4)$:
$$
f=((3,5),(2,5),(2,4),(1,2)),
$$
or the factorization 
$[4,1,5,3,2]=(3,5)(2,5)(2,4)(1,2)$ 
\end{itemize}

\end{example}

\section{Uniform bounds for unitary Weingarten functions}
\label{section:uniform_bounds}

\subsection{Main results}

Our main estimate is as follows: 
\begin{theorem}\label{combinatorial-estimate}
Let $k$ be a positive integer. 
For any permutation $\sigma \in \Sy{k}$ and 
nonnegative integer $g$, we have
$$
 (k-1)^g \# P (\sigma , |\sigma|) \leq 
\# P(\sigma , |\sigma |+2g)\leq 
(6 k^{7/2})^g \# P (\sigma , |\sigma|).
$$
\end{theorem}

\begin{theorem}\label{wg-estimate}
For any $\sigma\in \Sy{k}$ and  $d > \sqrt{6} k^{7/4}$, 
$$
\frac{1}{1-\frac{k-1}{d^2}} \le \frac{(-1)^{|\sigma|} d^{k+|\sigma|} 
\Wg^{\U} (\sigma,d)}
{\# P(\sigma,|\sigma|)} \le \frac{1}{1- \frac{6 k^{7/2}}{d^2}}.
$$
In addition, the l.h.s inequality is valid for any 
$d \ge k$.
\end{theorem}

\begin{proof}
It follows Theorems \ref{thm:unitary-Wg-expansion2} and 
\ref{combinatorial-estimate} immediately.
\end{proof}

\subsection{Comments}
\begin{enumerate}
\item
This bound implies that 
$$d^{k+|\sigma |}
\Wg^{\U} (\sigma, d)\to
\mathrm{Moeb}(\sigma) (:=
 (-1)^{|\sigma |}\# P(\sigma,|\sigma|))$$ 
in $d \to \infty$
for any given $\sigma$.
This was long known. 
There was also a need for uniform bounds for 
theoretical purposes, and has actually already had many applications in QIT. 
Some weaker bounds have been obtained by 
Montanaro \cite{montanaro} and by Collins et al. \cite{CGGPG}.
See also \cite{BG15}.

These two bounds are actually not comparable (one is better than the other depending on the nature of $\sigma$ -- specifically, 
its distance to the identity). 
The bound that we present in this manuscript is optimal within a polynomial factor, and 
it improves simultaneously on the two previous bounds. 
\item
We believe (because of the full cycle) that the optimal ratio is 
$1-\frac{k^3}{3d^2}$. 
Indeed, in the case of the full cycle $Z_k$ in $\Sy{k}$, we know that
$$\Wg^{\U} (Z_k,d)=\frac{\mathrm{Cat}(k-1)}{(d-k+1)\ldots (d+k-1)}.$$
Expanding the denominator shows that 
$\Wg^{\U} (Z_k,d)\sim  \mathrm{Cat}(k-1) d^{-2k+1}$ as soon as $k^3/d^2\to 0$.
This would be 
reminiscent of universality (cf for example \cite{soshnikov}).
Indeed, in many occurrences of random matrix theory, the largest eigenvalue
of eigenvalues has fluctuations of the order $d^{-2/3}$ and they can be analyzed
through moments methods with moments that grow as the dimension to the power
$3/2$. This is exactly the phenomenon that we are witnessing here. 

As for us, we just obtained our result for $k^{7/2}/d^2\to 0$, however, we believe that 
$k^3/d^2\to 0$ is the right bound, and leave it as an open question.
\end{enumerate}

\subsection{The proof of Theorem \ref{combinatorial-estimate}}\label{subsec:proof}

\subsubsection{The easy estimate}

In order to obtain the left estimate of Theorem \ref{combinatorial-estimate},
it is enough to show the inequality 
$$
(k-1) \#P(\sigma, l) \le \# P(\sigma,l+2)
$$
for any $l \ge |\sigma|$.
Consider a path $p=(\sigma_0,\dots, \sigma_{k+l}) \in P(\sigma,l)$ and
a transposition $\tau$ of the form $\tau=(i,k)$ with $1 \le i \le k-1$.
Then the sequence
$$
\tilde{p}=(\sigma_{-2}, \sigma_{-1}, \sigma_0,\dots, \sigma_{k+l})
\qquad \text{with $\sigma_{-2}=\sigma$ and $\sigma_{-1}= \tau \sigma$}
$$
is a path from $\sigma$ to $\emptyset$, going through $l+2$ solid edges,
i.e., $\tilde{p} \in P(\sigma, l+2)$. 
The map $(\tau, p) \mapsto \tilde{p}$ is clearly injective.
This fact gives the desired inequality.

\subsubsection{Estimates for Catalan numbers}

\begin{lemma}
\label{lemma:Matsumoto-Novak}
For $\sigma \in \Sy{k}$ with cycle-type 
$\mu=(\mu_1,\mu_2,\dots)$,
$$
\# P(\sigma,|\sigma|) = 
\prod_{i=1}^{\ell(\mu)} \mathrm{Cat}(\mu_i-1),
$$
where $\mathrm{Cat}(n)
= \frac{(2n)!}{(n+1)! \, n!}$ is the $n$-th Catalan number.  
\end{lemma}

\begin{proof}
It is known  that $\#\mathcal{F}(\sigma,|\sigma|)= \prod_{i=1}^{\ell(\mu)} \mathrm{Cat}(\mu_i-1)$, 
see \cite[Corollary 2.11]{matsumoto-novak}.
We have the result from Lemma \ref{lemma:PvsF} immediately. 
\end{proof}

\begin{lemma}  \label{lemma:neiborhood}
Let $\sigma \in \Sy{k}$, and let $\tau$ 
be a transposition in $\Sy{k}$.
Then 
$$
\# P(\tau \sigma,|\tau\sigma|)
\le 6 k^{3/2} \# P(\sigma,|\sigma|).
$$
\end{lemma}

\begin{proof}
Let $\mu=(\mu_1,\mu_2,\dots)$ be the cycle-type of $\sigma$.
Then it is well known that the cycle-type of $\tau \sigma$
is obtained from $\mu$ by a {\em cut operation} or a  {\em join operation}.
By a cut operation, a part $\mu_r$ (greater than $1$) is decomposed into two parts $(i,j)$
for some positive integers $i,j$ with $i+j=\mu_r$.
By a join operation, two parts $\mu_r, \mu_s$ are combined as
$\mu_r+\mu_s$. 
Therefore,
together with Lemma \ref{lemma:Matsumoto-Novak},
we find that
the ratio 
$\frac{\# P(\tau \sigma,|\tau\sigma|)}
{\# P( \sigma,|\sigma|)}$
is bounded by
$$
\max_{r+s +2 \le k} \left\{
\frac{\mathrm{Cat}(r+s+1)}{
\mathrm{Cat}(r) \mathrm{Cat}(s)}, \ 
\frac{\mathrm{Cat}(r) \mathrm{Cat}(s)}
{\mathrm{Cat}(r+s+1)}
\right\}.
$$
It is clear that $\frac{\mathrm{Cat}(r) \mathrm{Cat}(s)}
{\mathrm{Cat}(r+s+1)} \le 1$
because of the recurrence formula
$\mathrm{Cat}(n+1) = \sum_{i+j=n} \mathrm{Cat}(i) \mathrm{Cat}(j)$.

Let us estimate the ratio $\frac{\mathrm{Cat}(r+s+1)}{
\mathrm{Cat}(r) \mathrm{Cat}(s)}$.
Using the Stirling's formula with precise bounds \cite{robbins}
$$
\sqrt{2 \pi} n^{n+1/2} e^{-n} \le n! \le
e n^{n+1/2} e^{-n},
$$
we have inequalities for Catalan numbers
$$
\mathrm{Cat}(n)= \frac{1}{n+1} \frac{(2n)!}{(n!)^2} 
\le \frac{1}{n} \frac{e (2n)^{2n+1/2} e^{-2n}}
{(\sqrt{2\pi} n^{n+1/2} e^{-n})^2} = 
\frac{e}{\sqrt{2} \pi} \cdot 4^n n^{-3/2} 
$$
and 
$$
\mathrm{Cat}(n) \ge \frac{1}{2n} \frac{(2n)!}{(n!)^2} 
\ge \frac{1}{2n} 
\frac{\sqrt{2\pi} (2n)^{2n+1/2} e^{-2n}}
{(e n^{n+1/2} e^{-n})^2} =
\frac{\sqrt{\pi}}{e^2} \cdot 4^n n^{-3/2}.
$$
Therefore we see that
$$
\frac{\mathrm{Cat}(r+s+1)}{
\mathrm{Cat}(r) \mathrm{Cat}(s)}
\le 
 \frac{ \frac{e}{\sqrt{2} \pi}4^{r+s+1} (r+s+1)^{-3/2}}{
 \frac{\sqrt{\pi}}{e^2} 4^{r} r^{-3/2} \cdot 
 \frac{\sqrt{\pi}}{e^2} 4^s s^{-3/2}} 
= \frac{ \sqrt{8} e^5}{\pi^2} \left( \frac{rs}{r+s+1}\right)^{3/2}.
$$
Under the condition $r+s \le k$,
this is clearly bounded by
$$
\frac{ \sqrt{8} e^5}{\pi^2} \left[\left( \frac{rs}{r+s}\right)^{3/2}
\right]_{r=s=\frac{k}{2}} =\frac{e^5}{\sqrt{8} \pi^2} k^{3/2}.
$$
Here a numerical estimate gives $\frac{e^5}{\sqrt{8} \pi^2} =5.31...$.
\end{proof}

\subsubsection{Deep observations for $P(\sigma, l)$}

Let us recall 
the Weingarten graph
$\mathcal{G}^{\U}$ defined in \S \ref{subsection:Weingarten-graphs}.
Consider a vertex $\sigma \in \Sy{k}$ and 
a path 
$p=(\sigma_0, \dots, \sigma_{k+l}) \in P(\sigma,l)$.
The path $p$ goes through exactly $l$ solid edges and
$k$ dashed edges.
Furthermore, we see that
\begin{itemize}
\item If $\sigma_i \rightarrow \sigma_{i+1}$, then
$\sigma_i, \sigma_{i+1}$ have the same level $t$ with some $t \in \{2,3,\dots,k\}$
and satisfy the relation $\sigma_{i+1}= (s,t) \sigma_i$ with some $s \in \{1,2,\dots,t-1\}$. Moreover, we have $|\sigma_{i+1}|=|\sigma_i| \pm 1$.
\item If $\sigma_i \dashrightarrow \sigma_{i+1}$, then $\sigma_{i+1}=(\sigma_i)^{\downarrow}$ and 
the level of $\sigma_{i+1}$ is smaller by $1$ than that of $\sigma_i$.
Moreover, $|\sigma_i|=|\sigma_{i+1}|$.
\end{itemize}
If $l=|\sigma|$, then  $|\sigma_i| > |\sigma_{i+1}|$ whenever $\sigma_{i} \rightarrow \sigma_{i+1}$. 

From now on, we assume $l>|\sigma|$.
Then there exist  solid edges  $\sigma_{i} \rightarrow \sigma_{i+1}$ 
satisfying $|\sigma_i|<|\sigma_{i+1}|$.
We write $j(p):=j$ if the $(j+1)$-th solid edge in $p$ is the first one among
them.
Since $|\sigma| < k$ for all $\sigma \in \Sy{k}$, the number $j(p)$ should be in
$\{0, 1,2,\dots,k-2\}$.
Let $\sigma_{j+k-r} \rightarrow \sigma_{j+k-r+1}$ be the present
$(j+1)$-th solid edge in $p$.
Then, the part $(\sigma_0,\dots,\sigma_{j+k-r})$ of $p$ 
goes through $j$ solid edge and $k-r$ dashed edges,
and therefore $\sigma_{j+k-r}$ and $\sigma_{j+k-r+1}$ are of level $r$.
Furthermore, the definition of $j=j(p)$ implies that
$|\sigma_{j+k-r}|=|\sigma|-j$ and $|\sigma_{j+k-r+1}|=|\sigma_{j+k-r}|+1$.

We now put
$$
P_j(\sigma,l) = \{p \in P(\sigma, l) \ | \ 
j(p)=j\}
$$
for $j \in \{0, 1,2,\dots,k-2\}$.
We have the decomposition
$$
P(\sigma,l)= \bigcup_{j=0}^{k-2} P_j(\sigma,l).
$$
Let $2 \le r \le k$,  and let $\rho, \rho' \in \Sy{r}$
be two different permutations in $\Sy{r}$ connected by a solid edge.
We furthermore put
$$
P_j(\sigma,\rho, \rho', l) = \{
p=(\sigma_0,\dots,\sigma_{k+l}) \in P_j(\sigma,l) \ | \ 
\sigma_{j+k-r}=\rho, \ \sigma_{j+k-r+1}=\rho'\}.
$$
As we saw in the previous paragraph, 
this set is nonempty only if 
\begin{equation} \label{eq:condition_j_rho}
j= |\sigma|-|\rho| \qquad \text{and} \qquad
|\rho'|=|\rho|+1.
\end{equation}
We have thus  obtained the decomposition
$$
P_j(\sigma,l)= \bigcup_{r=2}^k \bigcup_{\rho \in S_r}
\bigcup_{\begin{subarray}{c} \rho' \in S_r \\ \rho \rightarrow \rho'
\end{subarray}} P_j(\sigma,\rho,\rho',l).
$$

Let us consider each  set 
$P_j(\sigma,\rho,\rho',l)$ with $l=|\sigma|+2g$
and suppose that it is nonempty.
Decompose each path 
$$
p=(\sigma_0,\dots,\sigma_{j+k-r}, \sigma_{
j+k-r+1}, \dots, \sigma_{k+l}) \in P_j(\sigma,\rho,\rho',l)
$$ 
into two parts 
$q=(\sigma_0,\dots,\sigma_{j+k-r})$ and 
$q'=(\sigma_{
j+k-r+1}, \dots, \sigma_{k+l})$.
Then $q$ is a sequence (or a {\em partial path}) from
$\sigma=\sigma_0$ to $\sigma_{j+k-r}=\rho$,
going through
$j$ solid edges and $k-r$ dashed edges.
Also,  $q'$ is a path from $\rho'$ to $\emptyset$
going through $l-j-1$ solid edges and $r$ dashed edges.
Since $l-j-1=(|\sigma|+2g)-(|\sigma|-|\rho|)-1=
|\rho| +2g-1=|\rho'|+2g-2$
by \eqref{eq:condition_j_rho}, the path $q'$ belongs to 
$P(\rho', |\rho'|+2g-2)$.
We thus obtain the bijection
\begin{equation} \label{eq:decomp-Pj-rho}
P_j(\sigma,\rho,\rho',|\sigma|+2g) \cong
\tilde{P}_j(\sigma, \rho) \times 
P(\rho', |\rho'|+2g-2),
\end{equation}
where
$\tilde{P}_j(\sigma, \rho)$ is, by definition, the collection of
all {\em partial paths} 
$q=(\sigma_0,\dots, \sigma_{j+k-r})$ 
from $\sigma=\sigma_0$ to $\sigma_{j+k-r}=\rho$, 
going through
$j(=|\sigma|-|\rho|)$ solid edges and $k-r$ dashed edges.

\subsubsection{Proof of the right estimate in Theorem \ref{combinatorial-estimate}} 

We shall prove inequalities
$$
\# P(\sigma,|\sigma|+2g) 
\le (6 k^{7/2})^{g} \# P(\sigma,|\sigma|),
\qquad (\sigma \in \Sy{k})
$$
by induction on $g$.
Note that the case where $g=0$ is trivial.
Assume that $g>0$.
The induction hypothesis claims that,
for all $r \ge 1$ and 
for all $\eta \in \Sy{r}$, it holds that
\begin{equation} \label{eq:induction_assumption}
\# P(\eta,|\eta|+2g-2) 
\le (6 r^{7/2})^{g-1} \# P(\eta,|\eta|).
\end{equation}

Let $\sigma \in \Sy{k}$ and 
consider $P(\sigma, |\sigma|+2g)$.
The cardinality of each 
nonempty subset $P_j(\sigma,\rho, \rho',|\sigma|+2g)$
is estimated as follows:
\begin{align*}
\# P_j(\sigma,\rho, \rho',|\sigma|+2g)
&= \# \tilde{P}_j(\sigma,\rho) \cdot \#
P(\rho', |\rho'|+2g-2) \\
&\le   \# \tilde{P}_j(\sigma,\rho) \cdot (6k^{7/2})^{g-1} \#
P(\rho',|\rho'|) \\
&\le   \# \tilde{P}_j(\sigma,\rho) \cdot (6k^{7/2})^{g-1} 
6k^{3/2} \cdot \#P(\rho,|\rho|).
\end{align*}
Here we have used \eqref{eq:decomp-Pj-rho}, 
\eqref{eq:induction_assumption}, and 
Lemma \ref{lemma:neiborhood}
in each step.
Together with the fact that, given $\rho \in \Sy{r}$, there are 
$r-1$ possibilities for $\rho'$, we obtain 
$$
\#\bigcup_{\rho'}  P_j(\sigma,\rho, \rho',|\sigma|+2g) \le 
\# \tilde{P}_j(\sigma,\rho) \cdot (6k^{7/2})^{g-1} 
6k^{3/2} \cdot k \cdot \#P(\rho,|\rho|).
$$
Here, since the natural map
\begin{align*}
& \bigcup_{r} \bigcup_{\rho \in \Sy{r}}( \tilde{P}_j(\sigma,\rho) \times P(\rho,|\rho|))  \to
 P(\sigma,|\sigma|): \\
& ( (\sigma,\sigma_1,\dots, \sigma_{j+k-r-1},\rho) ,
(\rho, \sigma_1',\dots, \sigma'_{|\rho|+r-1},\emptyset)) \\
& \qquad \mapsto 
(\underbrace{\sigma,
\sigma_1,\dots, \sigma_{j+k-r-1},
\rho,
\sigma_1',\dots, \sigma'_{|\rho|+r-1},\emptyset}_{\text{this has 
$|\sigma|$ solid edges and $k$ dashed edges}}),
\end{align*}
is well-defined and injective, 
summing over $\rho$'s,
we have
$$
\# P_j(\sigma,|\sigma|+2g) \le 
(6k^{7/2})^{g-1} 
6k^{5/2} \cdot \# P(\sigma,|\sigma|).
$$
Summing over $j$, one gets
$$
\# P(\sigma,|\sigma|+2g) \le 
(6k^{7/2})^{g-1} \cdot  6k^{5/2} \cdot k \cdot
\# P(\sigma,|\sigma|)
=(6k^{7/2})^{g}
\# P(\sigma,|\sigma|),
$$
as desired.
We have thus completed the proof of 
Theorem \ref{combinatorial-estimate}.

\section{Orthogonal groups}

In this section, we develop the case for orthogonal groups 
$$
\Ort (d) =\{ g  \in \GL (d,\mathbb{C}) \ | \  g g^{\mathrm{t}} = I_d \}.
$$
Most parts of the present section is same with the unitary case.

\subsection{Weingarten calculus} \label{sec:Wg-orthogonal}

Suppose that $d,k$ are positive integers with $d \ge 2k$. 
Let $\Pair{2k}$ be the set of pair partitions on $\{1,2,\dots, 2k\}$.
A pair partition $\mf{m} \in \Pair{2k}$ is expressed in the form
$$
\mf{m}= \{\mf{m}(1), \mf{m}(2) \} \{\mf{m}(3), \mf{m}(4)\} \cdots
\{\mf{m}(2k-1), \mf{m}(2k) \}.
$$
An ordered sequence $\mathbf{i}=(i_1,\dots,i_{2k})$ of $2k$ positive integers is called 
{\em admissible} for $\mf{m}$ if it holds that
$$
\{r,s \} \in \mf{m}  \quad \Rightarrow  \quad i_r =i_s.
$$
Furthermore, $\mathbf{i}$ is called {\em strongly admissible} for $\mf{m}$
if it holds that 
$$
\{r,s \} \in \mf{m}  \quad \Leftrightarrow  \quad i_r =i_s.
$$
For example, if $\mf{m}=\{1,3\}\{2,6\}\{4,5\}$ then
$(2,1,2,2,2,1)$ is admissible for $\mf{m}$ but not strongly admissible,
and $(2,1,2,3,3,1)$ is strongly admissible.

The symmetric group $\Sy{2k}$ acts transitively on $\mcal{P}_2 (2k)$ by
$$
\sigma . \mf{m}=\{ \sigma( \mf{m}(1)) , \sigma(\mf{m}(2)) \} 
\{\sigma(\mf{m}(3)), \sigma(\mf{m}(4)) \} \cdots
\{\sigma(\mf{m}(2k-1)), \sigma(\mf{m}(2k)) \}.
$$
In particular, we see that
$$
\sigma. \mf{e}= \{\sigma(1), \sigma(2)\}
\cdots \{\sigma(2k-1), \sigma(2k)\}
$$
for the ``trivial pair partition''
$$
\mf{e}=\mf{e}_k=\{1,2\}\{3,4\}\cdots\{2k-1,2k\}.
$$

For two pair partitions $\mf{m}, \mf{n}$ in $\Pair{2k}$,
we let $\mathbf{i}=(i_1,\dots,i_{2k})$ and 
$\mathbf{j}=(j_1,\dots,j_{2k})$ to be strongly admissible  for
$\mf{m}$ and $\mf{n}$, respectively. 
Then the {\em orthogonal Weingarten function} $\Wg^{\Ort}(\mf{m},\mf{n},d)$ is
defined by
\begin{align*}
\Wg^{\Ort}(\mf{m},\mf{n},d) =& \int_{\Ort (d)} 
u_{i_1 j_1} u_{i_2 j_2} \cdots u_{i_{2k} j_{2k}}
\  d\mu.
  \end{align*}
Here $d \mu=d\mu^{\Ort (d)}$ denotes the normalized Haar measure on $\Ort (d)$.
For example, if $\mf{m}=\{1,3\}\{2,6\}\{4,5\}$ 
and $\mf{n}=\{1,2\}\{3,4\}\{5,6\}$, then we can write
$$
\Wg^{\Ort}(\mf{m},\mf{n},d)
= \int_{\Ort (d)} u_{21}  u_{11} u_{22} u_{3 2} u_{3 3} u_{13} \, d\mu.
$$
By virtue of the bi-invariant property of the Haar measure,
this definition is independent of choices of strongly admissible sequences.

The Weingarten calculus for
orthogonal groups is stated as follows.

\begin{lemma}[\cite{collins-sniady}] \label{lemma:Orthogonal_Wg}
For two sequences 
$$
\mathbf{i}=(i_1,\dots,i_{2k}), \qquad \mathbf{j}=(j_1,\dots,j_{2k})
$$
of positive integers in $\{1,2,\dots,d\}$, we have
\begin{align*}
&\int_{\Ort (d)} u_{i_1 j_1} u_{i_2 j_2} \cdots u_{i_{2k} j_{2k}} \, d\mu  \\
&= \sum_{\mf{m} \in \Pair{2k}} \sum_{\mf{n} \in \Pair{2k}}
\Delta_{\mf{m}}(\mathbf{i}) \Delta_{\mf{n}}(\mathbf{j})
\Wg^{\Ort}(\mf{m},\mf{n},d).
\end{align*}
Here $\Delta_{\mf{m}}(\mathbf{i})$ is defined by
$$
\Delta_{\mf{m}}(\mathbf{i}) =
\begin{cases}
1 & \text{if $\mathbf{i}$ is admissible for $\mf{m}$}, \\
0 & \text{otherwise}.
\end{cases}
$$
\end{lemma}

We will use the following lemma later.

\begin{lemma} \label{lemma:WgO_invariant}
For any  $\sigma \in \Sy{2k}$ and $\mf{m},\mf{n} \in \Pair{2k}$, we have
$$
\Wg^{\Ort}(\sigma.\mf{m},\sigma.\mf{n},d) =\Wg^{\Ort}(\mf{m},\mf{n},d)
$$
\end{lemma}

\begin{proof}
It is easy to see that:
if $\mathbf{i}$ is a strongly admissible sequence for $\mf{m}$, then
$\mathbf{i}^{\sigma}=(i_{\sigma(1)},\dots,i_{\sigma(2k)})$
is strongly admissible for $\sigma.\mf{m}$.
Therefore
\begin{align*}
& \Wg^{\Ort}(\sigma.\mf{m},\sigma.\mf{n},d)
=  \int_{\Ort (d)} u_{i_{\sigma(1)},j_{\sigma(1)}} \cdots 
u_{i_{\sigma(2k)},j_{\sigma(2k)}} \, d\mu \\
=& \int_{\Ort (d)} u_{i_{1},j_{1}} \cdots 
u_{i_{2k},j_{2k}} \, d\mu =\Wg^{\Ort}(\mf{m},\mf{n},d).
\end{align*}
\end{proof}

We write 
$\Wg^{\Ort}(\mf{e},\mf{m},d)$ by $\Wg^{\Ort}(\mf{m},d)$,
which is also called  the orthogonal Weingarten function.
For each $\mf{m}$, there exist
some $\sigma_{\mf{m}} \in \Sy{2k}$
with $\sigma_{\mf{m}}. \mf{e}=\mf{m}$,
and  
$$\Wg^{\Ort}(\mf{m},\mf{n},d)=
\Wg^{\Ort} (\sigma_{\mf{m}}. \mf{e}, \mf{n},d)
= \Wg^{\Ort} (\mf{e}, \sigma_{\mf{m}}^{-1}.\mf{n},d) = 
\Wg^{\Ort} (\sigma_{\mf{m}}^{-1}.\mf{n},d)$$
by Lemma \ref{lemma:WgO_invariant}.
Thus, it is enough to deal with
the family
$\{\Wg^{\Ort}(\mf{m},d)\}_{\mf{m} \in 
\Pair{2k}}$ on behalf of
$\{\Wg^{\Ort}(\mf{m},\mf{n},d)\}_{\mf{m},\mf{n} \in 
\Pair{2k}}$.

\subsection{Orthogonality relations}

\begin{lemma} \label{lemma:Ort_rel_Ort}
The orthogonal Weingarten function satisfies the following 
formula:
For each $\mf{m} \in \Pair{2k}$,
\begin{equation} \label{eq:orth_ind}
d \Wg^{\Ort}(\mf{m},d) =
- \sum_{i=1}^{2k-2} \Wg^{\Ort} 
( (i, 2k-1). \mf{m},d) 
+\delta_{\{2k-1,2k\} \in \mf{m}} \Wg^{\Ort} (\mf{m}^{\downarrow},d).
\end{equation}
Here 
\begin{itemize}
\item $\delta_{\{2k-1, 2k\} \in \mf{m}} = \begin{cases}
1 & \text{if $\{2k-1,2k\}\in \mf{m}$}, \\
0 & \text{otherwise}.
\end{cases}$
\item $\mf{m}^{\downarrow}$ is 
the pair partition 
in $\Pair{2k-2}$ obtained from $\mf{m}$
by removing the block $\{2k-1,2k\}$ (if possible).
\end{itemize}
\end{lemma}

\begin{proof}
In the present proof, we abbreviate as
$\Wg(\mf{m},\mf{n})=\Wg^{\Ort}(\mf{m},\mf{n},d)$
and $\Wg(\mf{m})=\Wg^{\Ort}(\mf{m},d)$.
Let $\mathbf{i}=(i_1,\dots,i_{2k})$ be a strongly admissible 
sequence for $\mf{m}$.
Consider the sum of integrals
\begin{equation} \label{eq:sum_orthogonal}
\sum_{i=1}^d \int_{\Ort(d)} (\prod_{r=1}^{k-1} u_{i_{2r-1}, r} u_{i_{2r},r}) \cdot
 u_{i_{2k-1},i} u_{i_{2k},i} \, d\mu.
\end{equation}
Since $U =(u_{ij})$ is orthogonal, we have $\sum_{i=1}^d
 u_{i_{2k-1},i} u_{i_{2k},i}  =\delta_{i_{2k-1},i_{2k}}$,
which is equal to $\delta_{\{2k-1,2k\} \in \mf{m}}$.
Note that, if $\{2k-1,2k\} \in \mf{m}$, the sequence $(i_1,\dots,i_{2k-2})$ is strongly
admissible for $\mf{m}^{\downarrow}$.
Therefore \eqref{eq:sum_orthogonal} equals
$$
\delta_{\{2k-1,2k\} \in \mf{m}}
\int_{\Ort(d)} (\prod_{r=1}^{k-1} u_{i_{2r-1}, r} u_{i_{2r},r})\, d\mu
= \delta_{\{2k-1,2k\} \in \mf{m}} \Wg(\mf{m}^{\downarrow}).
$$

On the other hand, 
the sequence $\mathbf{j}=(1,1,2,2,\dots,k-1,k-1, i,i)$ is 
admissible for $\mf{e}$.  
Moreover, if $i <k$, the sequence $\mathbf{j}$ is admissible for 
$(2i-1, 2k-1).\mf{e}$ and $(2i, 2k-1). \mf{e}$.
Hence, it follows  from
Lemma \ref{lemma:Orthogonal_Wg} that:
\begin{itemize}
\item if $i \ge k$, then
	$$
	\int_{\Ort(d)} (\prod_{r=1}^{k-1} u_{i_{2r-1}, r} u_{i_{2r},r}) \cdot
  u_{i_{2k-1},i} u_{i_{2k},i} \, d\mu 
  =\Wg (\mf{m}, \mf{e}) =\Wg (\mf{m});
  $$
\item if $i <k$, then 
	\begin{align*}
	&\int_{\Ort(d)} (\prod_{r=1}^{k-1} u_{i_{2r-1}, r} u_{i_{2r},r}) \cdot
  u_{i_{2k-1},i} u_{i_{2k},i} \, d\mu \\
  &=\Wg (\mf{m},\mf{e})+ \Wg (\mf{m}, (2i-1,2k-1).\mf{e})
 + \Wg(\mf{m}, (2i,2k-1).\mf{e}) \\
 &= \Wg (\mf{m})+ \Wg( (2i-1,2k-1).\mf{m})
 + \Wg ((2i,2k-1).\mf{m}).
	\end{align*}	
\end{itemize}
Summing up over $i$, the equation \eqref{eq:sum_orthogonal} equals
$$
d \Wg(\mf{m})+ \sum_{j=1}^{2k-2} \Wg((j,2k-1). \mf{m}). 
$$
\end{proof}

\begin{example}
We again abbreviate as
$\Wg(\mf{m})=\Wg^{\Ort}(\mf{m},d)$.
Equation \eqref{eq:orth_ind} with $k=2$
gives three identities
\begin{align*}
d \Wg( \{1,2\}\{3,4\})=& 
- \Wg (\{1,3\}\{2,4\}) 
- \Wg (\{1,4\}\{2,3\}) 
+ \Wg( \{1,2\});  \\
d \Wg( \{1,3\}\{2,4\})=& 
- \Wg (\{1,2\}\{3,4\}) 
- \Wg (\{1,3\}\{2,4\}); \\
d \Wg( \{2,3\}\{1,4\})=& 
- \Wg(\{2,3\}\{1,4\}) 
- \Wg (\{1,2\}\{3,4\}). 
\end{align*}
The second and third identities imply that
$$
\Wg( \{1,3\}\{2,4\}) =
\Wg (\{2,3\}\{1,4\})= -\frac{1}{d+1} 
\Wg (\{1,2\}\{3,4\}).
$$
Furthermore, the first identity gives
\begin{align*}
d \Wg (\{1,2\}\{3,4\}) =& 
\frac{2}{d+1}\Wg (\{1,2\}\{3,4\}) 
+ \Wg( \{1,2\}),
\end{align*}
so that 
$$
\Wg (\{1,2\}\{3,4\}) = \frac{d+1}{(d+2)(d-1)} \cdot
\Wg (\{1,2\})= \frac{d+1}{(d+2) d(d-1)}.
$$
\end{example}

\subsection{Weingarten graphs} \label{subsection:Wg_graph_Ort}

The Weingarten graph for the orthogonal group
can be defined quite similar way to the unitary group.
We only mention the difference between them.
We consider the graph $\mathcal{G}^{\Ort}=(V,E)$.
The vertex set $V$ is $\bigsqcup_{k=0}^\infty \Pair{2k}$.
For convenience, we set $\Pair{0}= \{\emptyset\}$
with the `empty pair partition' $\emptyset=\mf{e}_0$.
Two vertices $\mf{m}$, $\mf{n}$ of level $k$ (i.e. in $\Pair{2k}$)
are connected by a {\em solid arrow} as $\mf{m} \rightarrow \mf{n}$
if and only if 
$$
\mf{n}= (i,2k-1). \mf{m} \qquad \text{with some $i$ smaller than $2k-1$}.
$$
For each $k \ge 1$, a vertex $\mf{m}$ of level $k$ and a vertex $\mf{m}'$
of level $k-1$ are connected by a {\em dashed arrow} as
$\mf{m} \dasharrow \mf{m}'$ if and only if 
$\{2k-1,2k\} \in \mf{m}$ and $\mf{m}^{\downarrow} = \mf{m}'$.

We also consider a path $p=(\mf{m}_0,\mf{m}_1,\dots, \mf{m}_{l+k})$
as in the unitary case.
Denote by $P(\mf{m},l)$ the collection of 
all paths from $\mf{m} \in \Pair{2k}$ to $\emptyset$,
going through $l$ solid edges and $k$ dashed edges.

\begin{lemma} 
For $\mf{m} \in \Pair{2k}$, we have the formal expansion
\begin{equation}
\Wg^{\Ort}(\mf{m},d)= d^{-k} \sum_{l \ge 0} \#P(\mf{m},l) (-d^{-1})^l.
\end{equation}
\end{lemma}

\begin{proof}
It is same with the proof of Lemma \ref{lemma:unitary-Wg-expansion1}.
Use Lemma \ref{lemma:Ort_rel_Ort}.
\end{proof}

For each pair partition $\mf{m} \in \Pair{2k}$, we 
associate with the {\em coset-type} $\mu$,
which is an integer partition of $k$. 
We can see its definition in e.g.  \cite{matsumoto-ramanujan, matsumoto-rm2}.
Put $|\mf{m}|=k-\ell(\mu)$. 
For example, 
$|\mf{m}|=0$ if and only if  $\mf{m}=\mf{e}_k$.

We can observe the fact $|\mf{m}^{\downarrow}| =|\mf{m}|$ and 
$$
|\tau. \mf{m}| -|\mf{m}| \in \{-1,0,1\}
$$
for any transposition $\tau$ in $\Sy{2k}$.
We emphasis that the equality 
$|\tau. \mf{m}| = |\mf{m}|$ may happen,
different from the unitary case.
We have obtained the following expansion.
Note that, different from the unitary case,
the summation of the right hand side is alternating.

\begin{theorem} \label{thm:ort_Wg_expansion2}
For each $\mf{m} \in \Pair{2k}$, we have the formal expansion
$$
(-1)^{|\mf{m}|} d^{|\mf{m}|+k} \Wg^{\Ort} (\mf{m},d)=
\sum_{g \ge 0} \# P(\mf{m},|\mf{m}|+g) (-d)^{-g}.
$$
\end{theorem}

\subsection{Monotone factorizations}

Consider a pair partition $\mf{m} \in \Pair{2k}$ and
a sequence $f=(\tau_1,\dots,\tau_l)$ of $l$ transpositions satisfying:
\begin{itemize}
\item $\tau_i=(s_i,2t_i-1)$ with $1 \le s_i < 2t_i-1 \le 2k-1$;
\item $k \ge t_1 \ge t_2 \ge \cdots \ge t_l \ge 1$;
\item $\mf{m}= (\tau_1 \cdots \tau_l). \mf{e}$.
\end{itemize}
Such a sequence $f$ is called a {\em monotone factorization} of length 
$l$ for $\mf{m}$.
We denote by $\mathcal{F}(\mf{m},l)$ the collection of these $f$.

\begin{lemma} \label{lemma:PvsF2}
Let $\mf{m} \in \Pair{2k}$.
For any nonnegative integer $l$, there exists
a $1$-to-$1$ correspondence between
$P(\mf{m},l)$ and $\mathcal{F}(\mf{m},l)$.
\end{lemma} 

\begin{proof}
It is the same with that of Lemma \ref{lemma:PvsF}.
\end{proof}

\subsection{Symplectic groups}

Consider the symplectic group
$$
\Smp(d)= \{ g \in \U (2d) \  | \ 
 g J  = J g\},
 \qquad \text{with} \quad
 J=J_d = \begin{pmatrix}  O_d & I_{d}\\
 -I_d & O_d 
 \end{pmatrix}.
$$
The Weingarten calculus for $\Smp(d)$ can be described in a similar way to
orthogonal groups $\Ort(d)$. 
We do not state the specific formula here 
and we are interested in only the absolute value of the Weingarten function.
The readers can see the exact formula in \cite{collins-stolz, matsumoto-rm2}.
For each pair partition $\mf{m} \in \Pair{2k}$, 
the symplectic Weingarten function 
$\Wg^{\Smp} (\mf{m},d)$ is given by 
$$
\pm \Wg^{\Smp} (\mf{m},d) = \Wg^{\Ort}(\mf{m},-2d),
$$
{\em up to sign}. 
Here the quantity $\Wg^{\Ort}(\mf{m},-2d)$
is obtained from the orthogonal Weingarten function $\Wg^{\Ort}(\mf{m},d)$
by replacing $d$ with $-2d$ formally. 
Therefore, from Theorem \ref{thm:ort_Wg_expansion2},
we can expand it 
 as follows.

\begin{theorem} \label{thm:sp_expansion}
For each $\mf{m} \in \Pair{2k}$, we have the formal expansion
$$
(2d)^{|\mf{m}|+k} |\Wg^{\Smp} (\mf{m},d)|
= \sum_{g \ge 0} \# P(\mf{m},|\mf{m}|+g) (2d)^{-g}.
$$
\end{theorem}

\subsection{Uniform bounds}

\begin{theorem} \label{combinatorial-estimate2}
Let $k$ be a positive integer.
For any pair partition $\mf{m} \in \Pair{2k}$ and
nonnegative integer $g$, we have
\begin{align}
\# P(\sigma, |\mf{m}|+2g) & \ge (2k-2)^g \# P(\mf{m},|\mf{m}|),  \notag \\
\# P(\sigma, |\mf{m}|+g) & \le (12k^{7/2})^g \# P(\mf{m},|\mf{m}|).
\label{combinatorial-estimate2-2}
\end{align}
\end{theorem}

\begin{theorem} \label{wg-symplectic-bound}
For any $\mf{m} \in \Pair{2k}$ and $d > 6 k^{7/2}$,
\begin{equation} \label{eq:symplecti-bound}
\frac{\#P(\mf{m},|\mf{m}|)}{1-\frac{k-1}{2d^2}} \le 
(2d)^{|\mf{m}|+k} |\Wg^{\Smp} (\mf{m},d)| \le
\frac{\# P(\mf{m},|\mf{m}|)}{1- \frac{6 k^{7/2}}{d}}.
\end{equation}
\end{theorem}

\begin{proof}
This is a direct consequence from Theorems
\ref{thm:sp_expansion} and \ref{combinatorial-estimate2}.
The left estimate is obtained by ignoring odd degree terms: 
$$
(2d)^{|\mf{m}|+k} |\Wg^{\Smp} (\mf{m},d)| \ge 
(2d)^{|\mf{m}|+k} \sum_{g \ge 0} \# P(\mf{m}, |\mf{m}|+2g) (2d)^{-2g}.
$$
\end{proof}

Since the orthogonal Weingarten function is an alternating sum (Theorem \ref{thm:ort_Wg_expansion2}),
we obtain a slightly weaker upper bound.
We can also obtain a lower bound, but due to the fact that the orthogonal case involves signed sums, it is not as sharp as in the unitary or symplectic case.

\begin{theorem} \label{wg-orthogonal-bound}
For any $\mf{m} \in \Pair{2k}$ and $d > 12 k^{7/2}$,
$$
\# P(\mf{m},|\mf{m}|)\frac{1- \frac{24 k^{7/2}}{d}}{1- \frac{144 k^7}{d^2}}
\le 
(-1)^{|\mf{m}|}d^{|\mf{m}|+k} \Wg^{\Ort} (\mf{m},d) \le
\frac{\# P(\mf{m},|\mf{m}|)}{1- \frac{144 k^7}{d^2}}.
$$
\end{theorem}

\begin{proof}
From Theorem \ref{thm:ort_Wg_expansion2} we see that 
the positive value $(-1)^{|\mf{m}|} d^{|\mf{m}|+k} \Wg^{\Ort} (\mf{m},d)$ is equal to
$$
\sum_{g \ge 0} \#P(\mf{m},|\mf{m}|+2g) d^{-2g} - 
\sum_{g \ge 0} \#P(\mf{m},|\mf{m}|+2g+1) d^{-(2g+1)}.
$$
Applying \eqref{combinatorial-estimate2-2}
to the first summand and ignoring the second summand, we have 
$$
(-1)^{|\mf{m}|} d^{|\mf{m}|+k} \Wg^{\Ort} (\mf{m},d) \le 
\# P(\mf{m},|\mf{m}|)\sum_{g \ge 0} c^{2g}= \frac{\# P(\mf{m},|\mf{m}|)}{1-c^2}
$$
with $c= 12k^{7/2} d^{-1} (<1)$.
On the other hand, applying \eqref{combinatorial-estimate2-2}
to the second summand and ignoring the first summand except the first term, we have
\begin{align*}
(-1)^{|\mf{m}|}d^{|\mf{m}|+k} \Wg^{\Ort} (\mf{m},d) & \ge 
\# P(\mf{m},|\mf{m}|)- 
\# P(\mf{m},|\mf{m}|)\sum_{g \ge 0} c^{2g+1} \\
&= \# P(\mf{m},|\mf{m}|) \left(1-\frac{c}{1-c^2} \right) \ge 
\# P(\mf{m},|\mf{m}|) \cdot \frac{1-2c}{1-c^2}.
\end{align*}
\end{proof}

\subsection{Proof of Theorem \ref{combinatorial-estimate2}}

The proof is obtained in a similar way to subsection \ref{subsec:proof}.
We only mention the difference between them.
Note that the first inequality in Theorem \ref{combinatorial-estimate2}
can be obtained in a similar way to the unitary case.
In fact, we can choose $2k-2$ solid edges connected with  $\mf{m} \in \Pair{2k}$.

Let us show the second inequality of the theorem.
First we observe the explicit value for 
$P(\mf{m},|\mf{m}|)$.
Recall that the cardinality of $P(\mf{m},l)$ for 
the graph $\mathcal{G}^{\Ort}$ is different from
that of $P(\sigma,l)$ for 
the graph $\mathcal{G}^{\U}$ in general.
Nevertheless, the numbers of {\em shortest paths} in each case
coincide.

\begin{lemma}
\label{lemma:Matsumoto}
For $\mf{m} \in \Pair{2k}$ with coset-type 
$\mu=(\mu_1,\mu_2,\dots)$,
$$
\# P(\mf{m},|\mf{m}|) = 
\prod_{i=1}^{\ell(\mu)} \mathrm{Cat}(\mu_i-1).
$$
Therefore, for any transposition $\tau$ in $\Sy{2k}$, we have
$$
\# P(\tau. \mf{m},|\tau. \mf{m}|)
\le 6 k^{3/2} \# P(\mf{m},|\mf{m}|).
$$
\end{lemma}

\begin{proof}
The first statement is seen in \cite[Theorem 5.4]{matsumoto-ramanujan}.
We also use Lemma \ref{lemma:PvsF2}.
The latter statement is shown in Lemma \ref{lemma:neiborhood}.
\end{proof}

Recall the Weingarten graph $\mathcal{G}^{\Ort}$.
Consider a vertex $\mf{m} \in \Pair{2k}$ and
a path $p=(\mf{m}_0,\dots, \mf{m}_{k+l}) \in P(\mf{m},l)$.
Suppose that $l>|\mf{m}|$ and that
two vertices $\mf{m}_i,\mf{m}_{i+1}$ in $p$ are connected 
by a solid arrow: 
 $\mf{m}_{i} \rightarrow \mf{m}_{i+1}$.
Different from the unitary case, it happens that
$|\mf{m}_{i+1}|-|\mf{m}_i|= -1$, $0$, or $+1$.
We write $j(p)=j$ if
the $(j+1)$-th solid edge $\mf{m}_{j+k-r} \rightarrow \mf{m}_{j+k-r+1}$ 
in $p$ (with some $r$)
is the first solid edge satisfying
$$
|\mf{m}_{i}| \le |\mf{m}_{i+1}|.
$$
The number $j(p)$ is well defined in $\{0,1,\dots, k-2\}$.
Then, the part $(\mf{m}_0,\dots,\mf{m}_{j+k-r})$ of $p$ 
goes through $j$ solid edges and $k-r$ dashed edges, and 
therefore  $\mf{m}_{j+k-r}$ and $\mf{m}_{j+k-r+1}$ are of levels $r$.
Furthermore,
$|\mf{m}_{j+k-r}| = |\mf{m}|-j$, and
$$
|\mf{m}_{j+k-r+1}|= |\mf{m}_{j+k-r}|  \qquad \text{or} \qquad
|\mf{m}_{j+k-r+1}|= |\mf{m}_{j+k-r}|+1.
$$
Like the unitary case, we define
$$
P_j(\mf{m},l) = \{p \in P(\mf{m}, l) \ | \ 
j(p)=j\}
$$
and 
$$
P_j(\mf{m},\mf{n}, \mf{n}', l) = \{
p=(\mf{m}_0,\dots,\mf{m}_{k+l}) \in P_j(\mf{m},l) \ | \ 
\mf{m}_{j+k-r}=\mf{n}, \ \mf{m}_{j+k-r+1}=\,\mf{n}'\}
$$
for pair partitions $\mf{n},\mf{n}'$.
This is nonempty only if
\begin{itemize}
\item $j=|\mf{m}|-|\mf{n}|$;
\item $\mf{n}$ and $\mf{n}'$ have the same level $r$ and are connected 
by a solid edge;
\item $|\mf{n}'| =|\mf{n}|$ or $|\mf{n}'| =|\mf{n}|+1$,
\end{itemize}
which should be compared with \eqref{eq:condition_j_rho}.

Let $g \ge 1$. We can obtain bijections
\begin{align*}
& P_j(\mf{m},\mf{n}, \mf{n}', |\mf{m}|+g) \\
& \cong 
\begin{cases}
\tilde{P}_j (\mf{m},\mf{n}) \times P(\mf{n}', |\mf{n}'|+g-2) &
\text{if $|\mf{n}'|=|\mf{n}|+1$}, \\
\tilde{P}_j (\mf{m},\mf{n}) \times P(\mf{n}', |\mf{n}'|+g-1)
&
\text{if $|\mf{n}'|=|\mf{n}|$,}
\end{cases}
\end{align*}
where
$\tilde{P}_j(\mf{m}, \mf{n})$ is, by definition, the collection of
all {\em partial paths} 
$q=(\mf{m}_0,\dots, \mf{m}_{j+k-r})$ 
from $\mf{m}=\mf{m}_0$ to $\mf{m}_{j+k-r}=\mf{n}$, 
going through
$j(=|\mf{m}|-|\mf{n}|)$ solid edges and $k-r$ dashed edges.

We shall show the second inequality in Theorem 
\ref{combinatorial-estimate2} by induction on $g$.
Using the induction assumption, 
a similar discussion to the unitary case gives
$$
\# P_j(\mf{m},\mf{n},\mf{n}',|\mf{m}|+g) \le 
\# \tilde{P}_j (\mf{m},\mf{n}) \cdot
(12 k^{7/2})^{g-1} \cdot 6 k^{3/2} \cdot \# P(\mf{n},|\mf{n}|).
$$
Together with the fact that, given $\mf{n} \in \Pair{2r}$,
there are $2r-2$ possibilities for $\mf{n}'$ (since 
$\mf{n}, \mf{n}'$ are connected by a solid edge), 
we obtain
$$
\# \bigcup_{\mf{n}'} P_j(\mf{m},\mf{n},\mf{n}',|\mf{m}|+g) \le 
\# \tilde{P}_j (\mf{m},\mf{n}) \cdot
(12 k^{7/2})^{g-1} \cdot 6 k^{3/2} \cdot 2k \cdot \# P(\mf{n},|\mf{n}|).
$$
The remaining discussion is same with the unitary case again.

\subsection{Discussion for 
right estimates}

In the left estimate of \eqref{eq:symplecti-bound},
we ignored the odd-degree terms.
If ones want to find a sharper estimate, 
we need to compare 
$\#P(\mf{m},|\mf{m}|+1)$ with
$\#P(\mf{m},|\mf{m}|)$.
In the present short subsection, we  observe the difficulty of
a direct comparison.

Let us recall an analogue of 
Lemma \ref{lemma:Matsumoto} for 
$\#P(\mf{m},|\mf{m}|+1)$.
If $I=(i_1,\dots,i_r)$ is a sequence 
of nonnegative integers, let us define
$\mathcal{D}_I$ as the set of Dyck paths 
of length $|I|:=i_1+\cdots+i_r$
whose height after $i_1, i_1+i_2,\dots$ steps is zero.
For each Dyck path $c \in \mathcal{D}_I$,
we denote by $\mathcal{A}(c)$
the area under $c$.
For example, for 
the Dyck path $c=(+1,-1,+1,-1) \in 
\mathcal{D}_{(2,2)}$, 
the area is 
$\mathcal{A}(c)=2$,
which is a sum of two triangles.

\begin{lemma}[partially conjectured by Matsumoto \cite{matsumoto-ramanujan} and proved by F\'{e}ray
\cite{feray-anncomb}]
If $\mu=(\mu_1,\dots,\mu_l)$ is the coset-type of 
$\mf{m} \in \Pair{2k}$, then we have the expression
$$
\#P(\mf{m},|\mf{m}|+1)= \sum_{c \in \mathcal{D}_{I_\mu}}
\mathcal{A}(c) =\sum_{i=1}^{\ell} 
\left( \sum_{c \in \mathcal{D}_{(\mu_i-1)}} \mathcal{A}(c) \right)
\prod_{j \not=i} \mathrm{Cat}(\mu_j-1)
$$
with $I_\mu=(\mu_1-1,\dots,\mu_l-1)$.
\end{lemma}

Lemma \ref{lemma:Matsumoto} states the formula
$\#P(\mf{m},|\mf{m}|)= |\mathcal{D}_{I_\mu}|=\sum_{c \in \mathcal{D}_{I_\mu}}1$.
If $\mu \not=(1^k)$, then we obtain a trivial inequality 
\begin{align*}
\# P(\mf{m},|\mf{m}|+1) &
 \ge \left( \sum_{c \in \mathcal{D}_{(\mu_1-1)}} \mathcal{A}(c) \right)
\prod_{j \ge 2} \mathrm{Cat}(\mu_j-1) \\
& \ge \prod_{j \ge 1} \mathrm{Cat}(\mu_j-1) \\
& = \# P(\mf{m},|\mf{m}|).
\end{align*}
However, if $\mu=(1^k)$, then
$$
\# P(\mf{m},|\mf{m}|+1)=0 \qquad \text{and} \qquad
\# P(\mf{m},|\mf{m}|)=1.
$$
Thus, it is not clear to find a uniform estimate 
between
$\# P(\mf{m},|\mf{m}|+1)$ and $\# P(\mf{m},|\mf{m}|)$.

\section{Compact symmetric spaces}

Let $G/K$ be a classical compact symmetric 
space. We may assume that
$G$ is a compact matrix group and 
$K$ is a closed subgroup fixed by 
a so-called Cartan involution $\theta$ of $G$.
Then the space $G/K$ is 
identified with the subset 
$\mathbb{S}= \{g \theta (g)^{-1}
 \ | \ g \in G\}$
of $G$.
The group $G$ acts on $\mathbb{S}$ by 
$g.s = g s \theta(g)^{-1}$ 
$(g \in G, \ s \in \mathbb{S})$.
It is known that 
there exists the unique probability measure
$d\nu$ on $\mathbb{S}$, which is invariant 
under this action.
E. Cartan \cite{cartan} classified classical compact symmetric spaces into seven series,
which are labelled as A~I, A~II, A~III, BD~I, C~I, C~II, and D~III.

The Weingarten calculus of $G/K$ is 
the method for computations of integrals of the forms
$$
\int_{\mathbb{S}} s_{i_1 j_1} \cdots 
s_{i_k j_k} \, d\nu \qquad 
\text{or} \qquad 
\int_{\mathbb{S}} s_{i_1 j_1} \cdots 
s_{i_k j_k}
\overline{s_{i_1' j_1'} \cdots 
s_{i_l' j_l'}}
 \, d\nu,
$$
where $s_{ij}: \mathbb{S} \to \mathbb{C}$ is the $ij$-coordinate function. 
This is 
arose in \cite{collins-stolz} and 
much developed in  \cite{matsumoto-rm2} by applying harmonic analysis of symmetric groups.

In this section, we focus on two symmetric
spaces of types A~I and A~III.
As we did for unitary groups and 
orthogonal groups,
we will find orthogonal relations
for Weingarten functions of those types.
The main results are the following.
\begin{itemize}
\item For type A~I. We will recover the result
in \cite{matsumoto-rm1} in a a simpler way, which claims 
that the Weingarten function of type A~I 
is essentially same with the 
orthogonal Weingarten function.
\item For type A~III. We will obtain 
a combinatorial expansion for 
the A~III Weingarten function,
as like Theorems \ref{thm:unitary-Wg-expansion2}, \ref{thm:ort_Wg_expansion2}.
\end{itemize}
Our technique can be applied for 
compact symmetric spaces 
of remaining types A~II, CI, $\dots$.
However, there are additional technicalities that are intrinsic to any given type. 
Therefore, in this paper, for the sake of brevity, but yet show the power of 
the original Weingarten approach, we stick to two types.
We expect to handle other types in subsequent research.  

\subsection{AI case: COE}

Consider the compact symmetric space 
$\U (d)/ \Ort(d)$.
Then the set $\mathbb{S}=
\COE (d)$ consists of
all $d \times d$
symmetric unitary matrices.
The random matrix ensemble
$\{\COE (d), d\nu)\}_{d \ge 1}$
is referred to the circular orthogonal ensemble
(COE).

Assume $d \ge 2k$. The Weingarten calculus for the COE 
is described as follows.
For each pair partition $\mf{m} \in \Pair{2k}$, 
we define the Weingarten function 
$$
\Wg^{\COE} (\mf{m},d)= \int_{\mathrm{COE}(d)}
\prod_{j=1}^k s_{2j-1,2j}  \cdot \overline{\prod_{\{a,b\} \in \mf{m}} s_{a,b}} \, d\nu.
$$
Note that $s_{a,b}=s_{b,a}$ 
since a matrix in $\mathrm{COE}$ is symmetric.
Moreover, for each permutation 
$\sigma \in \Sy{2k}$, we put
$$
\Wg^{\mathrm{COE}}(\sigma,d) = \Wg^{\mathrm{COE}} (\sigma. \mf{e}_{k},d),
$$
where $\sigma.\mf{e}_k$ is 
the pair partition $\{\sigma(1),\sigma(2)\} \cdots \{\sigma(2k-1),\sigma(2k)\}$.

\begin{lemma}[\cite{matsumoto-rm1}] 
\label{lemma:COE-WC}
For two seqeunces
$$
\mathbf{i}=(i_1,\dots,i_{2k}), \qquad \mathbf{j}=(j_1,\dots,j_{2k})
$$
of positive integers in $\{1,2,\dots,d\}$, we have 
$$
\int_{\COE (d)} s_{i_1, i_2} \cdots s_{i_{2k-1}, i_{2k}}
\overline{s_{j_1, j_2} \cdots s_{j_{2k-1}, j_{2k}}} \, d\nu=
\sum_{\sigma \in \Sy{2k}} \delta_{\sigma}(\mathbf{i}, \mathbf{j}) 
\Wg^{\COE}(\sigma,d).
$$
Here the $\delta$-symbol is defined as in Lemma \ref{lemma:unitaryWC}.
\end{lemma}

\begin{lemma} \label{lemma:COE_Wg_invariant}
For any $\sigma, \zeta \in \Sy{2k}$, we have
$$
\Wg^{\COE}(\zeta^{-1} \sigma,d)=
\int_{\COE(d)} s_{\zeta(1), \zeta(2)} \cdots 
s_{\zeta(2k-1) \zeta(2k)}
\overline{s_{\sigma(1),\sigma(2)} 
\cdots s_{\sigma(2k-1), \sigma(2k)}} \, d\nu. 
$$
\end{lemma}

\begin{proof} 
Recall the fact that the probability measure $d\nu$ is invariant under the action of $\U (d)$.
Since any permutation matrix is 
unitary, integrals are invariant
under the replacement 
$(s_{ij})_{1 \le i,j \le d} \mapsto 
(s_{\zeta(i), \zeta(j)})_{1 \le i,j \le d}$.
Thus we have the indentity
$$
\Wg^{\COE}(\sigma, d)= 
\int_{\COE(d)} s_{\zeta(1), \zeta(2)} \cdots 
s_{\zeta(2k-1) \zeta(2k)}
\overline{s_{\zeta \sigma(1),\zeta \sigma(2)} 
\cdots s_{\zeta\sigma(2k-1), \zeta\sigma(2k)}} \, d\nu.
$$
Replacing $\sigma$ by $\zeta^{-1} \sigma$,
we obtain the desired formula.
\end{proof}

We have the following orthogonality relation for $\Wg^{\COE}$.

\begin{lemma}
The Weingarten function $\Wg^{\COE}$ satisfies the following formula:
For each $\mf{m} \in \Pair{2k}$,
\begin{align*}
&(d+1) \Wg^{\COE}(\mf{m},d) \\
=& -\sum_{i=1}^{2k-2} \Wg^{\COE}((i,2k-1).\mf{m},d)
+ \delta_{\{2k-1,2k\}\in \mf{m}} \Wg^{\COE}(\mf{m}^{\downarrow},d).
\end{align*}
\end{lemma}

\begin{proof}
Consider a pair partition 
$\mf{m}=\{\mf{m}(1),\mf{m}(2)\} \cdots \{\mf{m}(2k-1), \mf{m}(2k)\}$
and suppose $\mf{m}(2k-1)=2k-1$.
Fix such an expression of $\mf{m}$ and 
let $\sigma_{\mf{m}}$ be the permutation $j \mapsto \mf{m}(j)$.
Consider the sum of integrals
\begin{equation}
\sum_{i=1}^d J_i(\mf{m}) \label{eq:sum_COE}
\end{equation}
with
\begin{align*}
J_i(\mf{m})=
\sum_{i=1}^d \int_{\COE (d)} 
&s_{1,2} \cdots s_{2k-3,2k-2} s_{i, 2k}  \\
& \times \overline{s_{\mf{m}(1), \mf{m}(2)} \cdots s_{\mf{m}(2k-3), \mf{m}(2k-2)}
s_{i, \mf{m}(2k)}}
\, d\nu. 
\end{align*}
Since a matrix in $\COE(d)$ is unitary, we have 
$\sum_{i=1}^d s_{i,2k} \overline{s_{i, \mf{m}(2k)}}
= \delta_{\mf{m}(2k), 2k}$, and hence \eqref{eq:sum_COE} equals
\begin{align*}
& \delta_{\mf{m}(2k),2k} \int_{\COE (d)} 
s_{1,2} \cdots s_{2k-3,2k-2} \overline{s_{\mf{m}(1), \mf{m}(2)} \cdots s_{\mf{m}(2k-3), \mf{m}(2k-2)}}
\, ds \\
=& \delta_{\{2k-1,2k\} \in \mf{m}} \Wg^{\COE}( \mf{m}^{\downarrow}, d).
\end{align*}

On the other hand, 
using Lemma \ref{lemma:COE-WC} we have
\begin{align*}
J_i(\mf{m})=& \begin{cases}
\Wg^{\COE} (\sigma_{\mf{m}}, d) & 
\text{if $i \not\in \{1,2,\dots,2k-2,2k\}$} \\
\Wg^{\COE} (\sigma_{\mf{m}},d) + \Wg^{\COE}(\sigma[\mf{m}, i],d) & \text{
if $i \in \{1,2,\dots,2k-2,2k\}$},
\end{cases}
\end{align*}
where $\sigma[\mf{m},i]$ is the permutation defined by
\begin{align*}
\sigma[\mf{m},i] =&
\begin{pmatrix} 
1 & \cdots & r-1 & r & r+1 & \cdots & 2k-1 & 2k \\
\mf{m}(1) & \dots & \mf{m}(r-1) & 2k-1 & \mf{m}(r+1) &
\cdots & i & \mf{m}(2k)
\end{pmatrix} \\
=& (i, 2k-1) \sigma_{\mf{m}}
\end{align*}
with $r \in \{1,2,\dots, 2k-2,2k\}$ uniquely determined by $\sigma_{\mf{m}}(r)=i$.
If $i=2k$, we can observe  
$$
\Wg^{\COE}((2k-1,2k) \sigma_{\mf{m}},d)= \Wg^{\COE}( \sigma_{\mf{m}},d)
$$
by Lemma \ref{lemma:COE_Wg_invariant}.
Summing up them over $i$, we have obtained
\begin{align*}
\eqref{eq:sum_COE}=& (d+1) \Wg^{\COE}(\sigma_{\mf{m}},d)+ \sum_{i=1}^{2k-2} 
\Wg^{\COE} ((i,2k-1) \sigma_{\mf{m}},d) \\
=& (d+1) \Wg^{\COE}(\mf{m},d)+ \sum_{i=1}^{2k-2} 
\Wg^{\COE} ((i,2k-1). \mf{m},d).
\end{align*}
\end{proof}

Comparing this lemma with Lemma \ref{lemma:Ort_rel_Ort},
we find the fact that
the orthogonality relation for $\Wg^{\COE}(\mf{m},d)$
coincides with that for $\Wg^{\Ort}(\mf{m},d+1)$
in association with the shift for $d$.
This induces the following theorem immediately.

\begin{theorem} \label{thm:COE}
Suppose $d \ge 2k$. For any $\mf{m} \in \Pair{2k}$, we have
$$
\Wg^{\COE}(\mf{m},d) = \Wg^{\Ort}(\mf{m},d+1).
$$
\end{theorem}

This theorem was first discovered in \cite{matsumoto-rm1} by
applying harmonic analysis of symmetric groups.
In our present proof, we could avoid those involving algebraic discussions.

\subsection{AIII case}

Let $a,b$ be positive integers.
Put $d =a+b$ and set
$$
d^-=a-b.
$$
Let us consider the compact symmetric space
$\U(d )/(\U(a) \times \U(b))$ of type A III.
The corresponding involution $\theta$ and 
matrix space $\mathbb{S}$ are 
$\theta(g)= I'_{ab} g I'_{ab}$ and
$\mathbb{S}=\{g I'_{ab} g^* I'_{ab} \ | \ 
g \in \U(d)\}$, respectively.
Here we set
$$
I_{ab}'=\mathrm{diag}(\underbrace{1,1,\dots,1}_a,
\underbrace{-1,-1,\dots,-1}_b).
$$
For convenience, we deal with 
$$
\tilde{\mathbb{S}}= \tilde{\mathbb{S}}(d,d^-)= \{s= gI'_{ab} g^* \ | \ 
g \in \U(d)\}
$$
instead of $\mathbb{S}$.
Any matrix in $\tilde{\mathbb{S}}$ is 
unitary and Hermitian.
The induced probability measure $d\nu$
on $\tilde{\mathbb{S}}$ is invariant under
the action 
$$
G \times \tilde{\mathbb{S}} \ni
(g_0, s) \mapsto g_0 s g_0^* \in \tilde{\mathbb{S}}.
$$

Suppose that $d \ge k$.
The A III Weingarten function is defined 
by
\begin{equation}
\Wg^{\mathrm{A \,III}} (\sigma,d,d^-)= \int_{\tilde{\mathbb{S}}(d,d^-)} s_{1\sigma(1)} s_{2\sigma(2) }
\cdots s_{k \sigma(k)} \, d\nu
\end{equation}
for $\sigma \in \Sy{k}$.
This is a conjugacy-invariant function 
on $\Sy{k}$.

\begin{lemma}[\cite{matsumoto-rm2}] \label{wg-AIII}
For two sequences 
$$
\mathbf{i}=(i_1,\dots,i_k), \qquad
\mathbf{j}=(j_1,\dots,j_k)
$$
of positive integers in $\{1,2,\dots,d\}$, 
we have 
\begin{equation}
\int_{\tilde{\mathbb{S}}(d,d^-)} s_{i_1 j_1} s_{i_2 j_2} \cdots s_{i_k j_k} \, d\nu
= \sum_{\sigma \in \Sy{k}} \delta_{\sigma}(\mathbf{i},\mathbf{j}) \Wg^{\mathrm{A\,III}}
(\sigma, d,d^-).
\end{equation}
\end{lemma}

We introduce an operation $\sigma \mapsto \sigma^{\flat}$ as follows.
Let $k \ge 2$ and 
suppose that $\sigma \in \Sy{k}$ satisfies 
$\sigma(k)=:r \not=k$ and $\sigma(r)=k$.
In other words, the letter $k$ belongs to a $2$-cycle in $\sigma$.
If we remove the $2$-cycle $(r,k)$ from $\sigma$, the output is
a bijection on $T(r,k)=\{1,2,\dots,r-1,r+1,\dots,k-1\}$.
We then define the permutation $\sigma^{\flat}$ in $\Sy{k-2}$ by 
$\sigma^{\flat} = \iota_r \circ \sigma|_{T(r,k)} \circ \iota_r^{-1}$
with the order-preserved bijection 
$$
\iota_r:T(r,k) \to \{1,2,\dots,k-2\}.
$$
For example, if $\sigma \in \Sy{5}$ is
$$
\sigma= \begin{pmatrix} 1 & 2 & 3 & 4 & 5 \\
4 & 5 & 1 & 3 & 2 
\end{pmatrix}
$$
in the two-row notation, we have $r=2$ and  
$$
\sigma^{\flat} = \begin{pmatrix} 1 & 2 & 3  \\
3 & 1 & 2  
\end{pmatrix} \in \Sy{3}.
$$

\begin{lemma} 
\label{lemma:AIII_orthogonal}
Let $\sigma \in \Sy{k}$.
\begin{align*}
d \Wg^{\mathrm{A\, III}}(\sigma,d,d^-)= 
&-\sum_{i=1}^{k-1} \Wg^{\mathrm{A\, III}}((i,k) \sigma,d,d^-)
 \\
& + \delta_{\sigma(k)=k} d^- \Wg^{\mathrm{A\, III}}(\sigma^{\downarrow},d,d^-) 
\notag \\
& + \delta_{(\sigma(k),k) \in C(\sigma)} 
\Wg^{\mathrm{A\, III}}(\sigma^{\flat},d,d^-). \notag
\end{align*}
Here,
if $\sigma(k) \not=k$ and $\sigma^2(k)=k$, i.e, if $k$ belongs to a $2$-cycle of $\sigma$, then
we set
$\delta_{(\sigma(k),k) \in C(\sigma)}=1$.
\end{lemma}

\begin{proof}
First we assume $\sigma(k)=k$ and consider 
\begin{equation}  \label{eq:sum_AIII1}
\sum_{i=1}^d \int_{\tilde{\mathbb{S}}} s_{1 \sigma(1)}
s_{2 \sigma(2)} \cdots s_{k-1, \sigma(k-1)} s_{ii} \, d\nu.
\end{equation}
Since 
$\sum_{i=1}^d s_{ii}=\Tr (s)= \Tr (I'_{ab})= a-b=d^-$, the sum \eqref{eq:sum_AIII1}
equals 
$d^- \Wg (\sigma^{\downarrow}, d, d^-)$.
On the other hand, as in the unitary case,
we see that 
\begin{align*}
& \int_{\tilde{\mathbb{S}}} s_{1\sigma(1)} \cdots s_{k-1, \sigma(k-1)} s_{ii} \, d\nu \\
=&
\begin{cases}
 \Wg^{\mathrm{A\, III}}(\sigma,d,d^-) & \text{if $i \ge k$}, \\
 \Wg^{\mathrm{A\, III}}(\sigma,d,d^-) + 
 \Wg^{\mathrm{A\, III}}((i,k)\sigma,d,d^-) & \text{if $i < k$}
\end{cases}
\end{align*}
by Lemma \ref{wg-AIII},
and hence 
$$
\eqref{eq:sum_AIII1} = d  \Wg^{\mathrm{A\, III}}(\sigma,d,d^-) + 
\sum_{i=1}^{k-1} \Wg^{\mathrm{A\, III}}((i,k)\sigma,d,d^-).
$$
Thus we have obtained the desired equality 
for the case where $\sigma(k)=k$.

Next we assume $\sigma(k) \not=k$ and let $r=\sigma^{-1}(k) \in \{1,2,\dots,k-1\}$.
Consider 
\begin{equation}\label{eq:sum_AIII2}
\sum_{i=1}^d \int_{\tilde{\mathbb{S}}} s_{1\sigma(1)} \cdots 
s_{r-1, \sigma(r-1)} s_{r,i} s_{r+1, \sigma(r+1)} \cdots s_{k-1, \sigma(k-1)}
s_{i,\sigma(k)} \, d\nu.
\end{equation}
(Note that each term is obtained from
$s_{1\sigma(1)} \cdots s_{rk} \cdots s_{k\sigma(k)}$ by replacing
two $k$'s with $i$.)
Since
\begin{equation*} 
\sum_{i=1}^d s_{r i} s_{i t} = 
\sum_{i=1}^d s_{ri} \overline{s_{ti}} = \delta_{rt}
\end{equation*}
the sum \eqref{eq:sum_AIII2} equals 
\begin{align*}
&\delta_{r\sigma(k)} 
\int_{\tilde{\mathbb{S}}}
s_{1\sigma(1)} \cdots 
s_{r-1, \sigma(r-1)}  s_{r+1, \sigma(r+1)} \cdots s_{k-1, \sigma(k-1)} \, d\nu \\
=&  \delta_{(r,k) \in C(\sigma)} 
\Wg^{\mathrm{A\, III}}(\sigma^{\flat},d,d^-)
\end{align*}
by the definition of $\sigma^{\flat}$.
On the other hand,
it is easy to see that 
\eqref{eq:sum_AIII2} equals 
$$
d  \Wg^{\mathrm{A\, III}}(\sigma,d,d^-) + 
\sum_{i=1}^{k-1} \Wg^{\mathrm{A\, III}}((i,k)\sigma,d,d^-)
$$
by Lemma \ref{wg-AIII} again.
This completes the proof of the lemma.
\end{proof}

Let us consider the Weingarten graph 
$\mathcal{G}^{\mathrm{A \, III}}=(V,E)$ 
of type  A III.
\begin{itemize}
\item The vertex set $V$ is $\bigsqcup_{k=0}^\infty \Sy{k}$.
Each vertex $\sigma$ in $\Sy{k}$ is said to be of  {\it level} $k$.
\item For each $k \ge 2$, 
two vertices $\sigma, \tau$ of level $k$ are connected by a solid arrow if
$$
\tau = (i,k) \sigma \qquad \text{with some $i$ smaller than $k$}.
$$
We write $\sigma \rightarrow \tau$.
\item For each $k \ge 1$, a vertex $\sigma$ of level $k$
and a vertex $\sigma'$ of level $k-1$ are connected by a dashed arrow
if $\sigma(k)=k$ and 
$\sigma'=\sigma^{\downarrow}$.  
We write $\sigma \dasharrow \sigma'$.
\item For each $k \ge 2$, a vertex $\sigma$ of level $k$
and a vertex $\sigma''$ of level $k-2$ are connected by a {\em squiggled} arrow
if $\sigma(k) =:r \not=k$ and $\sigma (r)=k$, and moreover $\sigma'' = \sigma^{\flat}$. 
We write $\sigma \rightsquigarrow \sigma''$.
\end{itemize}
For example, as already observed, we have an squiggled arrow
$$
\Sy{5} \ni \begin{pmatrix} 1 & 2 & 3 & 4 & 5 \\
4 & 5 & 1 & 3 & 2 
\end{pmatrix} \ \rightsquigarrow \ \begin{pmatrix} 1 & 2 & 3  \\
3 & 1 & 2  
\end{pmatrix}\in \Sy{3}.
$$

A vertex $\sigma \in \Sy{k}$ has exactly $k-1$ solid edges and
at most $1$ dashed edge and at most $1$ squiggled  edge.
No vertex has both dashed edges and squiggled edges.

Let $\sigma \in \Sy{k}$ and 
consider a {\it path} $p$ 
from $\sigma$ to $\emptyset$
in $\mathcal{G}^{\mathrm{A\, III}}$ as usual.
But in this case, there are squiggled edges.
Denote by 
$$
\ell_0(p), \qquad \ell_1(p), \qquad \ell_2(p)
$$ the numbers of 
solid/dashed/squiggled
edges such that $p$ gets through, respectively. 
It is clear that
$$
\ell_1(p)+2 \ell_2(p)=k.
$$
Put $\ell(p)=\ell_0(p)+\ell_1(p)+\ell_2(p)$.

Due to Lemma \ref{lemma:AIII_orthogonal}
it is not difficult to see
the following theorem.

\begin{theorem} \label{thm:AIII}
Let $\sigma \in \mf{S}_k$. Then
$$
\Wg^{\mathrm{A \, III}} (\sigma,d,d^-)= \sum_{p: \sigma \to \emptyset}
(-1)^{\ell_0(p)}
(d^-)^{\ell_1(p)} d^{-
\ell(p)}.
$$
\end{theorem}

\begin{example}
Consider $\sigma=[2,1] \in \Sy{2}$.
There are two kinds of paths from $\sigma$ to
$\emptyset$:
\begin{align*}
& [2,1] \rightarrow [1,2] \rightarrow [2,1] 
\rightarrow [1,2] \rightarrow \cdots \rightarrow [2,1] \rightsquigarrow \emptyset, 
\\
& [2,1] \rightarrow [1,2] \rightarrow [2,1] 
\rightarrow [1,2] \rightarrow \cdots \rightarrow [1,2] \dasharrow [1] \dasharrow
 \emptyset.
\end{align*}
The first one contributes to  the term 
$d^{-(2l+1)}$ and 
the second one contributed to $-(d^{-1})^2 d^{-(2l+3)}$. Thus we have the expansion
$$
\Wg^{\mathrm{A \, III}} ([2,1],d,d^-)=
\sum_{l \ge 0} d^{-(2l+1)} + \sum_{l \ge 0}(-1)
(d^{-})^2 d^{-(2l+3)}
$$
which equals $\frac{d^2-(d^-)^2}{d(d^2-1)}$.
\end{example}

In \cite{matsumoto-rm2}, 
we obtained the Fourier expansion
$$
\Wg^{\mathrm{A\, III}}(\sigma,d,d^-)= \frac{1}{k!}
\sum_{\lambda \vdash k} \frac{s_{\lambda}(1^a, (-1)^b)}{s_\lambda (1^d)} \chi^\lambda(\sigma),
$$
where $s_\lambda=s_{\lambda}(x_1,\dots,x_d)$ is the Schur function
and $\chi^\lambda$ is the irreducible character
of symmetric groups.
We know the formula 
$s_{\lambda} (1^d)= \frac{f^\lambda}{k!} \prod_{(i,j) \in \lambda}
(d+j-i) $, but there is no such closed formula for 
$s_{\lambda}(\underbrace{1,1,\dots,1}_a,
\underbrace{-1,-1,\dots,-1}_b)$.
Theorem \ref{thm:AIII} gives a new combinatorial expression
for the A~III Weingarten function.

\end{document}